\documentclass[12pt]{article}
\usepackage{geometry}
\usepackage{soul}
\usepackage{a4wide}
\usepackage{amssymb, amsmath,amsthm}
\usepackage{xcolor}
\usepackage[utf8]{inputenc}
\usepackage{enumitem}
\usepackage{amsmath}
\usepackage{graphicx}
\renewcommand{\d}{\partial}

\newtheorem{theorem}{Theorem}[section]
\newtheorem{corollary}{Corollary}[section]

\newtheorem{assumption}{Assumption}[section]
\newtheorem{condition}{Condition}[section]
\newtheorem{proposition}{Proposition}[section]
\newtheorem{conjecture}{Conjecture}[section]

\renewcommand{\P}{\mathbb{P}}
\newcommand{\Prob}{\mathbb{P}}
\newcommand{\E}{\mathbb{E}}
\newcommand{\R}{\mathbb{R}}
\newcommand{\Pmus}{\mathcal{P}(\mu,s,\varphi)}
\newcommand{\Pmup}{\mathcal{P}(\mu,s^p)}
\newcommand{\Pmusigma}{\mathcal{P}(\mu,\sigma^2)}
\newcommand{\Pmud}{\mathcal{P}(\mu,d)}
\newcommand{\OPT}{\text{OPT}}
\newcommand{\REV}{\text{REV}}

\newcommand{\blu}[1]{\textcolor{black}{{\sf{#1}}}}
\newcommand{\blue}[1]{\textcolor{black}{#1}}

\title{Distribution-free  expectation operators for robust pricing and stocking with heavy-tailed demand}
\author{
Pieter Kleer, Johan S.H. van Leeuwaarden and Bas Verseveldt\\\
Tilburg University\\
Department of Econometrics and Operations Research\\
\texttt{\{p.s.kleer,j.s.h.vanleeuwaarden,} \\
\texttt{b.verseveldt\}@tilburguniversity.edu}
}
\date{\today}

\begin{document}

\maketitle

\begin{abstract}
We obtain distribution-free bounds for various fundamental quantities used in probability theory by solving optimization problems that search for extreme distributions among all distributions with the same mean and dispersion. These sharpest possible bounds depend only on the mean and dispersion of the driving random variable. We solve the optimization problems by a novel yet elementary technique that reduces the set of all candidate solutions to two-point distributions. We consider  a general dispersion measure, with variance, mean absolute deviation and power moments as special cases. We apply the bounds to robust newsvendor stocking and monopoly pricing, generalizing foundational mean-variance works. This shows how pricing and order decisions respond to increased demand uncertainty, including scenarios where dispersion information allows for heavy-tailed demand distributions.
\end{abstract}


\section{Introduction}
Consider a random variable of which  only the mean $\mu$ and variance $\sigma^2$ are known. In statistics and probability theory it is common practice to assume that this random variable follows a normal distribution with the same mean and variance. In some situations such normal approximations can prove poor. Take for instance the rule for normal distributions that approximately 95\% of all values fall within two standard deviations, {i.e, $2\sigma$}. This rule fails when the actual distribution would be heavy tailed, with much less of the distribution concentrated around the mean. To avoid such poor approximations, one could instead of assuming normality consider bounds and worst-case distributions. Let us mention Chebyshev's inequality as an example, which provides an upper bound on the tail probability of a random variable using only the mean and variance~\cite{bienayme1853considerations,chebyshev1867valeurs}. 
The one-sided version of Chebyshev's inequality, also known as Cantelli's inequality, says 
$\P ( X \geq t )\leq  {\sigma^2}/({\sigma^2 + \left(t - \mu\right)^2}).
$
This inequality cannot be improved, as the upper bound is attained for the distribution supported on the  points $\mu-\sigma^2/(t-\mu)$ and $t$, with probabilities $(t-\mu)^2/(\sigma^2+(t-\mu)^2)$ and $\sigma^2/(\sigma^2+(t-\mu)^2)$, respectively. 
This might be expected, as among all probability distributions with a given mean and variance, the most spread out one is a two-point distribution.

In this paper we obtain Chebyshev-type inequalities for a rich class of objectives
by solving optimization problems of the form
\begin{align}\label{richclass}
    \mbox{maximize } & \frac{\E(\psi_u(X))}{\E(\psi_\ell(X))} \\ \nonumber
    \mbox{subject to } & \E(X) = \mu \ {\rm and} \  \E(\varphi(X)) = s.
\end{align}
Here, the constraints say that the random variable
$X$ has mean $\mu$ and dispersion $s$, with $\varphi(\cdot)$ some dispersion measure that needs to be specified. 
The ratio objective in \eqref{richclass} allows for nonlinear objectives, most notably the conditional expectation. We will need some conditions on the functions $\psi_\ell$ and $\psi_u$ that are discussed in Section \ref{sec:framework}.
Observe that Cantelli's inequality is the solution to the special case of \eqref{richclass} with $\psi_u(\cdot)$ an indicator function, $\psi_\ell(\cdot)$ the constant function $1$, so that $\E(\psi_\ell(X)) = 1$, and $\varphi(\cdot )$ a quadratic function. 
We will solve \eqref{richclass} while imposing only mild conditions on the function $
\varphi(\cdot)$, so that solutions hold for a rich class of dispersion measures.

\vspace{.2cm}
\noindent{\bf Power deviation.} One dispersion measure of particular interest is the $p$th power deviation \cite{bickel2011descriptive}, defined as
$
(\E|X-\mu|^p)^{1/p}=:\tau_p(X).
$
Note that $p=2$ corresponds with standard deviation, and $p = 1$ with mean absolute deviation (MAD), two well-known dispersion measures. 
The $p$th power deviation of a fixed distribution increases in $p$, and for $p\in[1,2)$, a finite $\tau_p(X)$
does not require existence of a finite second moment $\E(X^2)$. Hence, for the range $p\in[1,2)$, the $p$th power deviation is not so much affected by large deviations from the mean, which makes it more appropriate as a dispersion measure than standard deviation $(p=2)$ when empirical data would suggest non-Gaussian features and outliers. 
{The dispersion measure $\tau_p(X)$ thus serves as a natural extension of the mean-variance setting, offering greater flexibility in shaping the ambiguity set. For $p > 2$, the set imposes restrictions leading to lighter tails, while for $p < 2$, it accommodates heavier-tailed distributions. 
Generally speaking, lowering $p$ permits more extreme outcomes and brings increased tail risk, while increasing $p$ concentrates outcomes more tightly around the mean, reducing exposure to tail events. 
Despite this added flexibility of the parameter $p$, the framework retains the simplicity of the $p = 2$ case, as it relies on just two moment constraints: the mean and a fractional moment determined by $p$. This simple structure produces extremal distributions—the solutions to the optimization problem~\eqref{richclass}—that remain two-point distributions with a mass point on either side of the mean. 
\blue{These two-point distributions also often tend to behave more extreme themselves, in the sense that either the right or left mass point moves away from the mean, when $p$ becomes smaller. This is in particular true for the choices of $\psi_u$ and $\psi_\ell$ that we consider in Section \ref{sec:fundamental_quantities}.}


One important motivation for this paper is to assess the impact of \blue{dispersion information (such as variance or fractional moments)} on robust bounds and robust decision making.   Another  motivation is to design a method for solving 
\eqref{richclass} for many combinations of the dispersion function $\varphi(\cdot)$ and the objective function.

\vspace{.2cm}
\noindent{\bf Semi-infinite fractional programs.}  Within the optimization literature, problems of type \eqref{richclass} are categorized as semi-infinite programs, as the space of measures contains infinitely many probability distributions, while there are only finitely many constraints. 

More specifically, \eqref{richclass} is a semi-infinite fractional program, as the objective function is a ratio of two functions, and hence a nonlinear function of the probability distribution \cite{ji2021data,hettich1993semi,liu2017distributionally}. There are well developed  approaches for solving such programs, for instance by first transforming (using e.g.~Charnes-Cooper transformations \cite{chen2005complex}) the non-linear function to a linear function, and then applying primal-dual methods for solving the resulting semi-infinite linear program; see \cite{shapiro2009lectures,shapiro2001duality} for general theory, and \cite{chen2022distribution,van2023generalized} for applications in pricing. 
Such semi-infinite linear programs can often be reduced to an equivalent finite program that yields the same optimal value (e.g.,~\cite{rogosinski1958moments,han2015convex}), but certainly for transformed fractional problems it cannot be expected that there are closed-form or insightful solutions. 
One of the goals in this paper is to obtain closed-form solutions by 
circumventing traditional optimization techniques altogether, and primal-dual methods in particular, and instead developing a novel reduction technique from first principles, using features of two-point distributions. 
For a class of semi-infinite linear (non-fractional) programs with mean-variance constraints, several works have considered reduction of three-point distributions to two-point distributions by assuming additional properties of the objective functions \cite{birge1991bounding,popescu2005semidefinite,van2023second}. Reduction to two-point distributions, however, is generally considered hard or impossible \cite{van2022distributionally,van2022mad,chen2022correction}. In particular, for the class of semi-infinite fractional programs \eqref{richclass} we have not seen such reductions in the literature. 

\blue{Our focus lies on fractional programs, which, due to the general dispersion measure, encompass a relevant class of models for decision-making under uncertainty. These models can be effectively solved using the reduction method introduced in this paper. Extending this class by considering additional moment constraints presents a promising avenue for future research. For such richer sets of moment information and linear (non-fractional) semi-infinite programs, Bertsimas and Popescu \cite{bertsimas2005optimal} developed duality theory to establish tight bounds on the probability that a given random variable belongs to a specified set. Building on these tools, Popescu \cite{popescu2005semidefinite} developed an optimization framework for computing tight bounds on functional expectations of random variables with general moment constraints, incorporating structural properties like symmetry and unimodality into the ambiguity sets through semidefinite and second-order conic optimization methods. While Popescu’s framework is broader in scope, as it accommodates a wider range of moment and other constraints, the key difference is that we focus on a fractional objective instead of a linear one. This fractional objective introduces complexities that can potentially be addressed by integrating Popescu’s framework with the Charnes-Cooper transformation or by extending the reduction techniques developed in this paper. We leave the exploration of this setting with more than two moment constraints as a promising direction for future research.}

\vspace{.2cm}
\noindent{\bf Our framework.} The solution method we present consists of two steps. We first formulate conditions on the dispersion and objective functions for which the set of all candidate solutions of \eqref{richclass} can be reduced to a set of two-point distributions. In the second step of our method we perform an optimization over this set of two-point distributions and solve \eqref{richclass}. 

This second step of the solution method depends in a subtle way on the specific choice of objective function. We execute this second step in full detail to obtain tight bounds for three special objective functions: the conditional expectation $\E(X|X\geq t)$, the tail probability $\P(X\geq t)$ and the max-operator (also known as expected loss function) $\E(\max\{X-t,0\})$. 
These three objectives each find many applications in the  literature, and in this paper we consider two such applications: the newsvendor model and monopoly pricing. 

\vspace{.2cm}
\noindent{\bf Applications.} We will use our distribution-free bounds to address robust versions of two well-known applications in operations research: The newsvendor model and pricing. Taking the best decision under distributional uncertainty is nowadays a large branch of operations research known as distributionally robust optimization (DRO). 
Next to inventory management \cite{Scarf1958, gallego1992minmax, perakis2008regret, ben2013robust}, DRO techniques have been applied to e.g.~scheduling \cite{kong2013scheduling,mak2014appointment}, portfolio optimization \cite{popescu2007robust,Delage2010}, insurance mathematics and option pricing 
\cite{lo1987semi,cox1991bounds},
monopoly pricing
\cite{elmachtoub2021value,chen2022distribution}, and stopping theory \cite{boshuizen1992moment,kleerleeuwaarden2022}.

The newsvendor model was introduced by Arrow et al.~\cite{arrow1951optimal} for finding the order quantity that minimizes expected costs in view of unknown demand and the trade-off between leftovers and lost sales. 
The mathematical analysis centers around the expectation $\E(\max\{0,D-q\})$ with $D$ the random demand and $q$ the order quantity. 
In traditional approaches, the demand distribution is fully specified, so that  $\E(\max\{0,D-q\})$ can be calculated, and the optimal order quantity can be determined. 
The standard reference for a distributionally robust approach is the work of Scarf \cite{Scarf1958}, who considered the newsvendor problem with mean-variance demand information, so a special case of the richer class of dispersion measures considered in this paper. Scarf obtained the tight bound on $\E(\max\{0,D-q\})$, and then solved a  minimax problem to find the optimal order quantity. We thus consider the counterpart that allows for more extreme demand scenarios, including heavy-tailed ones. 
A comparable setting with heavy-tailed demand is studied in \cite{das2021heavy}, where the ambiguity set contains all distributions with known first and the $\alpha$th moment with $\alpha>1$. For $\alpha< 2$ this allows heavy-tailed distribution with infinite second moment. There are some notable differences with \cite{das2021heavy}, who restrict to non-negative distributions, while we (just as Scarf \cite{scarf1958min}) put no further restrictions on the support. Further, \cite{das2021heavy} do not aim at solving exactly the optimization problem, presenting instead various approximative results for the optimal order quantity and how it responds to regularly varying  demand distributions. We base our analysis on exact solutions of the optimization problems for a general class of dispersion measures. In this way, we also find a new relation between the robust order quantity and the tail exponent of the worst-case demand distribution. 

The second application that we consider in more detail is from the area of pricing. We will consider the monopoly pricing problem, one of the cornerstones of revenue management. In the most basic setting, the monopolist knows the demand distribution and selects the price that strikes the right balance between margins and market share, and hence maximizes expected revenue; see e.g.~\cite{riley1983optimal} and \cite{myerson1981optimal}. We instead consider a robust version, where the monopolist only knows the mean and dispersion of demand, and then solves
the maximin problem that finds the revenue-maximizing price against the worst-case demand distribution. For the absolute revenue and mean-variance information this robust approach was pioneered by \cite{azar2013optimal}, and later generalized to other information sets \cite{carrasco2018optimal,kos2015selling,elmachtoub2021value,chen2021screening}
or objectives \cite{chen2022distribution,elmachtoub2021value}. As objective, we will consider the approximation ratio, which describes the relative difference between the robust revenue against worst-case demand and the maximum revenue that could
have been extracted in the full-information case.
Giannakopoulos et al.~\cite{giannakopoulos2023robust} recently obtained 
the robust
price that performs optimally in terms of the approximation ratio for mean-variance information. We generalize this result for variance to settings with general dispersion measures. As will become clear, solving the minimax problem for determining the optimal price will involve the tight bounds for the conditional expectation and the tail probability. 
\blue{We remark that in our robust pricing results we assume the distributions in the ambiguity set  to be supported on $\R$, but the same results remain true if we assume the distributions to have nonnegative support.}


\vspace{.2cm}
\noindent{\bf Contributions.}
Our contributions can be summarized as follows:
\begin{itemize}
\item We present a new reduction technique for solving robust optimization problems of 
the form \eqref{richclass}. This  leads to a class of optimization problems that are solved by two-point distributions. The resulting tight bounds are the sharpest possible bounds for settings where only the mean and dispersion of the underlying random variable are known. While the existing literature on comparable bounds focuses primarily on variance, we allow for a rich class of dispersion measures, include those that allow for heavy-tailed (infinite-variance) distributions. 

\item We then leverage this reduction technique to find tight bounds for three elementary quantities from probability theory: conditional expectation, tail probability and the expected loss function, generalizing existing results for mean-variance information \cite{mallows1969,scarf1958min,giannakopoulos2023robust}.
We apply these tight bounds to the newsvendor problem and monopoly pricing. In both cases, the robust bounds turn the classical settings into minimax optimization problems that can be solved explicitly. These minimax solutions provide guidance for making robust choices in response to limited information. We show how information affects the optimal order and pricing decisions, in particular when the dispersion information allows for heavy-tailed demand scenarios. 
\end{itemize}

\vspace{.2cm}
\noindent{\bf Outline.} In Section \ref{sec:framework}
we explain the key ideas behind the reduction technique. 
Section~\ref{sec:fundamental_quantities} uses the reduction technique to analyze three fundamental quantities from probability theory, leading to tight bounds for all distributions with given mean and dispersion. 
These three quantities are then applied in Section~\ref{sec:applications} for novel minimax analyses of the newsvendor model and the monopoly pricing problem. Some future directions are given in Section~\ref{conclus}.

\section{Reduction to two-point distributions}
\label{sec:framework}
We consider a so-called ambiguity set of (real-valued) distributions with given mean and dispersion, where the latter is modeled through a function $\varphi$. For $\mu, s > 0$, we define 
\begin{align}
\mathcal{P}(\mu,s,\varphi) = \{\Prob : \E_{\P}(X) = \mu, \text{ and } \E_{\P}(\varphi(X)) = s\},
    \label{eq:feasible}
\end{align}
where $X \sim \Prob$ is a real-valued random { variable} distributed according to distribution $\P$. We sometimes suppress the subscript $\Prob$ in the expectation operator $\E(\cdot)$.  For the case of dispersion based on the $p$th power deviation, we define the ambiguity set
\begin{align}
\mathcal{P}(\mu,s^p) = \{\Prob : \E_{\P}(X) = \mu \text{ and } \E_{\P}(|X-\mu|^p) = s^p\}
    \label{eq:feasible_p_power}
\end{align}
and for the special cases of $p = 1,2$, we use
\begin{align}
\mathcal{P}(\mu,\sigma^2) &= \{\Prob : \E_{\P}(X) = \mu \text{ and }  \E_{\P}((X-\mu)^2) = \sigma^2\},  \\
\mathcal{P}(\mu,d) &= \{\Prob : \E_{\P}(X) = \mu \text{ and } \E_{\P}(|X-\mu|) = d\}, 
    \label{eq:feasible_sigma}
\end{align}
which we refer to as mean-variance and mean-MAD ambiguity, respectively.

We are interested in solving the following problem
\begin{align}
    \displaystyle \max_{\Prob \in \Pmus} \frac{\E(\psi_u(X))}{\E(\psi_{\ell}(X))}
\label{eq:problem}
\end{align}
where $\psi_\ell, \psi_u : \R \rightarrow \R$. The goal of this paper is to show that  problem \eqref{eq:problem} can be solved by an optimization over two-point distributions under certain assumptions on $\psi_\ell, \psi_u$ and $\varphi$. In particular, we will need at least the following conditions.
\begin{assumption}[Conditions for $\psi_u$ and $\psi_\ell$]
\label{ass:psi}
There exists a value $t \in \R$ such that the function $\psi_u$ is 
    \begin{itemize}
     \item[{\rm (i)}] $\psi_u$ is concave and non-decreasing on $\{x : x < t\}$ and $\{x : x \geq t\}$;
          \item[{\rm (ii)}]  $\psi_l$ is convex and non-increasing on $\{x : x < t\}$ and $\{x : x \geq t\}$.
    \end{itemize}
\end{assumption}
The existence of $t$  as in Assumption \ref{ass:psi}  is known for many applications of problem \eqref{eq:problem}, as we will see in the next section.
\begin{assumption}[Conditions for $\varphi$]
\label{ass:phi}
The function $\varphi$ is
    \begin{itemize}
         \item[\rm (i)] strictly convex on $\R$;
         \item[\rm (ii)] differentiable on $\R \setminus \{\mu\}$;
         \item[\rm (iii)] 
 and satisfies $\lim_{|x| \rightarrow \infty} | {\varphi(x)}/{x}| = \infty.$
    \end{itemize}
\end{assumption}
We next make a simplifying assumption, namely that $\varphi(\mu) = 0$ and that $\varphi$ attains its global minimum at $\mu$. This is shown in Proposition \ref{prop:zero}.
\begin{proposition}
Under {\rm Assumption \ref{ass:phi}}, we can assume w.l.o.g. that $\varphi(\mu) = 0$ and that $\varphi$ attains its global minimum at $\mu$.
\label{prop:zero}
\end{proposition}
\begin{proof}
Choose $a \in \d\varphi(\mu)$ (where $\d\varphi(\mu)$ denotes the subdifferential) and define $\ell(\mu) = -a(x - \mu) - \varphi(\mu)$. Now $\bar{\varphi} = \varphi + \ell$ satisfies $\bar{\varphi}(\mu) = 0$, and since $\ell$ is a linear function the strict convexity of $\varphi$ implies that the function $\bar{\varphi}$ is strictly convex. Since $-a \in \d\ell(\mu)$ we have $0 \in \d\bar{\varphi}(\mu)$, and since $\bar{\varphi}$ is convex we conclude that $\varphi$ attains its global minimum at $\mu$. Since $\ell$ is differentiable on $\R$ and $\varphi$ on $\R - \{\mu\}$, we find that $\bar{\varphi}$ is differentiable on $\R - \{\mu\}$. The limits of $|x| \rightarrow \infty$ follow from the fact that $\lim_{|x| \rightarrow \infty}|\ell(x)/x| = |a|$, so the given limits for $\varphi$ combined with the triangle inequality yields the desired limits for $\bar{\varphi}$. { Hence $\bar{\varphi}$ also satisfies Assumption \ref{ass:phi}. 


 Now for any random variable $X$ we have $\E(\bar{\varphi}(X)) = \E(\varphi(X)) + \E(\ell(X)) = \E(\varphi(X)) - \varphi(\mu)$, so for an optimization problem with the constraint $\E(\varphi(X)) = s$ we can define an equivalent optimization problem with instead the constraint $\E(\bar{\varphi}(X)) = s - \varphi(\mu)$, where $\bar{\varphi}$ satisfies Assumption \ref{ass:phi} but also $\bar{\varphi}(\mu) = 0$ and $\bar{\varphi}$ attains its global minimum at $\mu$. Hence we may make these additional assumptions without loss of generality. }
\end{proof}

We will sketch the reduction to two-point distributions in the coming sections.

\subsection{Two-point distributions}
A two-point distribution, represented by the tuple $(v_1,w_1,v_2,w_2)$, is a probability distribution $\Prob$ supported on two points\footnote{{ If $v_1 = \mu$ then the distribution becomes a one-point distribution with support $\{\mu\}$. Because this trivial distribution is not in the ambiguity set for the problems we analyze, it is excluded.}} $v_1 < \mu < v_2$ with probability mass $w_1$ and $w_2$, respectively. In other words, if $X \sim \Prob$, then
\begin{align}
    X = \left\{ \begin{array}{ll}
         v_1 & \text{with prob. } w_1, \\
         v_2 & \text{with prob. } w_2.
    \end{array}\right. 
\end{align}
For two-point distributions, the feasible region $\Pmus$ can be described by  the equations
\begin{equation}
\begin{split}
 &   w_1 + w_2  = 1, w_1v_1 + w_2v_2  = \mu, \text{ and } w_1\varphi(v_1) + w_2\varphi(v_2)  = s
\end{split}
\label{eq:problem_2pt}
\end{equation}
where $w_1,w_2 \geq 0$ and $v_1 < \mu < v_2$.\\

We will first prove some useful properties of the feasible region $\Pmus$. 
Substituting the first and second constraint in the third one gives
\begin{align}
    f(v_1,v_2) := \frac{v_2 - \mu}{v_2 -v_1}\varphi(v_1) + \frac{\mu-v_1}{v_2 -v_1}\,\varphi\left(v_2 \right) = s.
    \label{eq:f}
\end{align}
     
An important property  of $f(v_1,v_2)$ is that it is non-decreasing in $v_2$.

\begin{proposition}
    Under {\rm Assumption \ref{ass:phi}(i)}, the function $f(v_1,v_2)$ as in \eqref{eq:f} is strictly increasing in $v_2$, i.e., $\partial f(v_1,v_2)/\partial v_2 > 0$. Furthermore, $\lim_{v_2 \rightarrow \infty} f(v_1,v_2) = \infty$.\footnote{This should be read as that $v_1$ is fixed, while $w_1$ and $w_2$ change as $v_2$ increases, in order to make sure that the first two constraints of \eqref{eq:problem_2pt} remain satisfied.}
    \label{prop:f_incr}
\end{proposition}
\begin{proof}
Note that
\begin{align*}
\partial f(v_1,v_2)/\partial v_2  & =  \frac{\mu - v_1}{(v_2 - v_1)^2}\varphi(v_1) - \frac{\mu - v_1}{(v_2 - v_1)^2}\varphi(v_2) + \frac{\mu - v_1}{v_2 - v_1}\varphi'(v_2) \\\
   & = \,  \frac{\mu - v_1}{(v_2 - v_1)^2}\big[\varphi(v_1) - \varphi(v_2) + (v_2 - v_1)\varphi'(v_2)\big] \\
    &= \,  \frac{\mu - v_1}{v_2 - v_1}\bigg[ \varphi'(v_2) - \frac{\varphi(v_2) - \varphi(v_1)}{v_2 - v_1}\bigg].
\end{align*}
Because $(\mu-v_1)/(v_2-v_1) > 0$, it suffices to show that 
    $$\varphi'(v_2) > (\varphi(v_2)-\varphi(v_1))/(v_2-v_1).$$
    This is true because $\varphi$ is assumed to be strictly convex. For the second claim note that
    \begin{align*}
    \lim_{v_2 \rightarrow \infty} f(v_1,v_2) &= \lim_{v_2 \rightarrow \infty} \frac{v_2 - \mu}{v_2 - v_1}\varphi(v_1) + \frac{\mu - v_1}{v_2 - v_1}\varphi(v_2)  \geq \lim_{v_2 \rightarrow \infty} \frac{v_2 - v_1}{v_2 - v_1}\varphi(v_1) + \frac{\mu - v_1}{v_2 - v_1}\varphi(v_2) \\
    & = \varphi(v_1) + \lim_{v_2 \rightarrow \infty} \frac{\mu - v_1}{v_2 - v_1}\varphi(v_2) 
     = \varphi(v_1) + (\mu - v_1)\lim_{v_2 \rightarrow \infty} \frac{v_2}{v_2 - v_1}\frac{\varphi(v_2)}{v_2} \\
    & = \varphi(v_1) + (\mu - v_1) \cdot 1 \cdot \infty 
     = \infty.
\end{align*}
This proves the claim.
\end{proof}

Next we show that, under Assumption \ref{ass:phi}, there exists a ``continuum" of two-point distributions which are feasible for \eqref{eq:problem_2pt}. We define this continuum in terms of $v_1$, but this can also be done using any of the other three parameters $w_1,v_2$ or $w_2$.

\begin{proposition}
    Under {\rm Assumption \ref{ass:psi}}, there exists for every $v \in (-\infty,\mu)$ a two-point distribution in $\Pmus$ supported on $v_1 = v$ and $v_2(v)$, with probability mass $w_1(v)$ and $w_2(v)$, respectively. 
    Furthermore, the function $v_2 : (-\infty,\mu) \rightarrow (\mu,\infty)$, mapping $v$ to $v_2(v)$ is increasing with $\lim_{v \rightarrow \mu} v_2(v) = \infty$.
    \label{prop:two_point}
\end{proposition}
\begin{proof}
    We define $f_v(v_2) = f(v,v_2)$ with $f$ as in \eqref{eq:f}.    The goal is to show that the equation $f_v(v_2) = s$ has a solution (for every fixed $v$). The probability masses then follow as $w_2 = (\mu-v)/(v_2 - v)$ and $w_1 = 1 - w_2$. From Proposition \ref{prop:f_incr} we know that $f_v(v_2)$ is strictly increasing in $v_2$. This implies that if the equation $f_v(v_2) = s$ has a solution $v_2$, this solution is unique.
We will next argue that a solution indeed exists.  First note that if $v_2 \rightarrow \mu$, then $f_v(v_2) \rightarrow 0 < s$ by assumption. By Assumption \ref{ass:phi}(iii) it follows that if $v_2 \rightarrow \infty$, then $f_v(v_2) \rightarrow \infty$. Because $\varphi$ is continuous on $(\mu,\infty)$, it follows by the {\it intermediate value theorem} that the equation $f_v(v_2) = s$ has at least one solution.

We continue with the second part of the statement. 
    We can also show that $v_2$ is increasing in $v$. For this, we take the derivative of $f(v_1,v_2)$ with respect to $v_1$, i.e., 
\begin{align*}
  \partial f(v_1,v_2)/\partial v_1   & = \frac{v_2 - \mu}{(v_2 - v_1)^2}\varphi(v_1) + \frac{v_2 - \mu}{v_2 - v_1}\varphi'(v_1) - \frac{v_2 - \mu}{(v_2 - v_1)^2}\varphi(v_2) \\
    & = \frac{v_2 - \mu}{(v_2 - v_1)}\left[\varphi'(v_1) - \frac{\varphi(v_2) - \varphi(v_1)}{v_2-v_1} \right] < 0,
\end{align*}
 since $\varphi$ is strictly convex (meaning the part in brackets is negative). Since we already know from Proposition \ref{prop:f_incr} that $f(v_1,v_2)$  is increasing with respect to $v_2$, we find that if $v_1$ increases then $v_2$ must also increase in order to make sure that the constraint $f(v_1,v_2) = s$ remains satisfied, and vice versa.

 Finally, as $v \rightarrow \mu$, then to make sure that $f(v_1,v_2) = s$ remains satisfied, it is necessary that $\varphi(v_2(v)) \rightarrow \infty$. Because of Assumption \ref{ass:phi}, it then follows that also $v_2(v) \rightarrow \infty$.
\end{proof}

\subsection{Reduction recipe}
Having established the existence of two-point distributions in the previous section, we next describe a generic procedure that shows that in order to solve  \eqref{eq:problem}, we can perform an optimization over two-point distributions. Firstly, for a given arbitrary distribution $\Prob$ satisfying the constraints in \eqref{eq:problem}, we define $\Prob_2$ as the distribution supported on $v_1$ and $v_2$ with corresponding probabilities $w_1$ and $w_2$, 
{ where}
\begin{equation}\label{eq:w1w2v1v2}
	w_1  = \P(X < t), \ \ 
	v_1  = \E(X \mid X < t), \ \ 
    w_2  = \P(X \geq t), \ \ \text{ and } \ \  
    v_2  = \E(X \mid X \geq t).
\end{equation}
If $X \sim \Prob$ and $Y \sim \Prob_2$, then the concavity of $\psi_u$ on $\{x : x < t\}$ and $\{x : x \geq t\}$ gives $\E_{\Prob}(\psi_u(X)) \leq \E_{\Prob_2}(\psi_u(Y))$, and similarly, $\E_{\Prob}(\psi_{\ell}(X)) \geq \E_{\Prob_2}(\psi_{\ell}(Y))$ by the convexity of $\psi_{\ell}$. These claims follow from an application of Jensen's inequality. This implies that
$$
\frac{\E(\psi_u(X))}{\E(\psi_{\ell}(X))} \leq \frac{\E(\psi_u(Y))}{\E(\psi_{\ell}(Y))}
$$
so the objective function of \eqref{eq:problem} does not decrease when replacing $\Prob$ by $\Prob_2$. 

Furthermore, by the convexity of $\varphi$, it follows that $\Prob_2$ satisfies the dispersion constraint with inequality, i.e., 
\begin{align}
\E_{\Prob_2}(\varphi(Y)) = f(v_1,v_2) \leq s,   
\label{eq:f2}
\end{align}
where $f(v_1,v_2)$ is as defined in \eqref{eq:f}. This also is because of Jensen's inequality. To summarize the reduction so far, in order to solve \eqref{eq:problem}, it suffices to solve
\begin{equation}
\begin{split}
\max &\, \frac{w_1\psi_u(v_1) + w_2 \psi_u(v_2)}{w_1\psi_{\ell}(v_1) + w_2 \psi_{\ell}(v_2)} \\
 \mbox{s.t. } &   w_1 + w_2  = 1, w_1v_1 + w_2v_2  = \mu, \text{ and } w_1\varphi(v_1) + w_2\varphi(v_2)   \leq s.
\end{split}
\label{eq:problem_2pt_full}
\end{equation}
In many cases, depending on the functions $\psi_u$ and $\psi_\ell$, it turns out that problem \eqref{eq:problem_2pt_full} has the same solution as the problem in which we set the dispersion constraint to hold with equality again. We will describe a sufficient condition for this that is satisfied by all the applications that we consider in Sections \ref{sec:fundamental_quantities} and \ref{sec:applications}.

From Proposition \ref{prop:f_incr} we know that $f$ is increasing in $v_2$, for fixed $v_1$. Hence if $v_1$ stays the same and $v_2$ increases, then the left hand side of the dispersion constraint in \eqref{eq:problem_2pt_full} increases. Hence, we could just increase $v_2$ until (\ref{eq:f2}) becomes an equality. However when doing this, we have to make sure that the objective function increases as well, which is a more subtle matter as $w_1$ and $w_2$ change as $v_2$ increases (while $v_1$ is being kept fixed). In particular we want the following condition to be satisfied. This condition is somewhat technical, but satisfied by many applications in the literature that we discuss in Section \ref{sec:applications}.

\begin{condition}
For every fixed $v_1 \in \{x : x < \mu\}$, the function
\begin{align}
g(v_1,v_2) = \frac{w_1(v_2)\psi_u(v_1) + w_2(v_2) \psi_u(v_2)}{w_1(v_2)\psi_{\ell}(v_1) + w_2(v_2) \psi_{\ell}(v_2)}
    \label{eq:cond_objective}
\end{align}
is non-decreasing in $v_2$, where $w_2(v_2) = (\mu - v_1)/(v_2 - v_1)$ and $w_1(v_2) = 1 - w_2(v_2)$.
\label{cond:objective}
\end{condition}
We can now summarize the whole reduction recipe in the following theorem.

\begin{theorem}
Assume that $\psi_\ell, \psi_u$ and $\varphi$ are such that {\rm Assumptions \ref{ass:psi}} and {\rm \ref{ass:phi}}, and {\rm Condition \ref{cond:objective}}, are satisfied. Then, solving \eqref{eq:problem} reduces to solving
\begin{equation}
\begin{split}
\displaystyle\max_{w_1,w_2 \geq 0, v_1 < \mu < v_2} \  & \frac{w_1\psi_u(v_1) + w_2 \psi_u(v_2)}{w_1\psi_{\ell}(v_1) + w_2 \psi_{\ell}(v_2)} \\
 \mathrm{s.t. \ \ \ \ \ \ \ } &   w_1 + w_2  = 1, \, w_1v_1 + w_2v_2  = \mu, \text{ and } w_1\varphi(v_1) + w_2\varphi(v_2)  = s.
\end{split}
\label{eq:problem_2pt_full_equal}
\end{equation}
    \label{thm:reduction_2pt}
\end{theorem}
\subsection{Mean absolute deviation}
The third condition in Assumption \ref{ass:phi} says that $\varphi$ should grow in a superlinear fashion as (the absolute value of) $x$ grows large. For mean absolute deviation (MAD) this third condition is violated. We will therefore treat MAD separately in this paper, also because of 
the important role of MAD in robust optimization studies that allow for heavy-tailed distributions, see e.g.~\cite{bental1972mad,roos2019chebyshev,elmachtoub2021value,chen2021screening}.
Since Assumption \ref{ass:phi} does not hold for MAD, we cannot use Propositions \ref{prop:f_incr} and \ref{prop:two_point} directly. Instead, we have the following result.
\begin{proposition}\label{prop:two_point_mad}
    For $\varphi(x) = |x - \mu|$, there exists for any $v \in (-\infty, \mu - d/2)$ a two-point distribution in $\Pmud$ supported on $v_1 = v$ and $v_2(v)$, with probability mass $w_1(v)$ and $w_2(v)$, respectively. Furthermore the function $v_2 : (-\infty,\mu - d/2) \rightarrow (\mu + d/2, \infty)$ mapping $v$ to $v_2(v)$ is increasing with $\lim_{v \rightarrow \mu - d/2} v_2(v) = \infty$. 
\end{proposition}
\begin{proof} For $\varphi(x) = |x - \mu|$, (\ref{eq:problem_2pt}) becomes
\begin{equation}
\begin{split}
 &   w_1 + w_2  = 1, \, w_1v_1 + w_2v_2  = \mu\, \text{ and } w_1(\mu - v_1) + w_2(v_2 - \mu)  = d
\end{split}
\end{equation}
where $w_1, w_2 \geq 0$ and $v_1 < \mu < v_2$. This can be solved exactly, with solution
\begin{equation}\label{eq:madsolution}
    w_1(v_1) = \frac{d}{2(\mu - v_1)}, \quad w_2(v_1) = 1 - \frac{d}{2(\mu - v_1)}, \quad v_2(v_1) = \mu + d/2 + \frac{d^2}{4\mu - 4v_1 - 2d}.
\end{equation}
Since $w_2 \geq 0$ we find $1 - \frac{s}{d(\mu - v_1)} \geq 0$, which can be rewritten as $v_1 \leq \mu - d/2$. Note that $v_1 = \mu - d/2$ would leave the system with no solution, hence we require $v_1 < \mu - d/2$. It is clear that for any $v_1 < \mu - d/2$ the constraints that $w_1, w_2 \geq 0$ are satisfied and that $v_2(v_1)$ is defined. Hence there exists a two-point distribution in $\Pmud$ supported on $v$ and $v_2(v)$ for any $v < \mu - d/2$. It is now easily verified that $v_2(v) > \mu + d/2$, that $v_2(v)$ increases in $v$ and that $\lim_{v \rightarrow \mu - d/2} v_2(v) = \infty$. 
\end{proof}
Note that Proposition~\ref{prop:two_point_mad} is very similar to Proposition \ref{prop:two_point}, but several key values are different, showing that MAD behaves differently than functions $\varphi$ that do satisfy Assumption \ref{ass:phi}. This suggests that first considering the $p$th power deviation with $p > 1$ and then taking the limit $p \rightarrow 1$ 
{ may return incorrect} results for MAD ($p=1$), and hence a separate treatment as in this section is necessary. 

Next, we would like to find a result similar to Theorem \ref{thm:reduction_2pt}. However, Condition \ref{cond:objective} is not sufficient in this case, since it may happen that $v_1 \geq \mu - d/2$. In that case, it is not possible to fix $v_1$ and increase $v_2$ until the constraint $w_1v_1 + w_2v_2 = d$ is satisfied. So we need a stronger condition that also allows for decreasing $v_1$ in order to satisfy this constraint.
\begin{condition}\label{cond:mad_condition}
    For every fixed $v_1 \in (-\infty,\mu-d/2)$, the function (\ref{eq:cond_objective}) is non-decreasing in $v_2$, where $w_2(v_2) = (\mu - v_1)/(v_2 - v_1)$ and $w_1(v_2) = 1 - w_2(v_2)$. Furthermore, for every fixed $v_2 \in (\mu+d/2,\infty)$, the function (\ref{eq:cond_objective}), is non-increasing in $v_1$, where again $w_2(v_2) = (\mu - v_1)/(v_2 - v_1)$ and $w_1(v_2) = 1 - w_2(v_2)$.
\end{condition}
We can now first decrease $v_1$ to a value lower than $\mu - d/2$, and then increase $v_2$ until we reach $w_1v_1 + w_2v_2 = d$. Using the same proof technique as for Theorem \ref{thm:reduction_2pt}, we find the following result.
\begin{theorem}
Assume that $\psi_\ell, \psi_u$ are such that {\rm Assumptions \ref{ass:psi}} and {\rm Condition \ref{cond:mad_condition}}, are satisfied. Then, solving \eqref{eq:problem} with $\varphi(x) = |x - \mu|$ reduces to solving
\begin{equation}
\begin{split}
\displaystyle\max_{w_1,w_2 \geq 0, v_1 < \mu-d/2, \mu+d/2 < v_2} \  & \frac{w_1\psi_u(v_1) + w_2 \psi_u(v_2)}{w_1\psi_{\ell}(v_1) + w_2 \psi_{\ell}(v_2)} \\
 \mathrm{s.t. \ \ \ \ \ \ \ } &   w_1 + w_2  = 1, \, w_1v_1 + w_2v_2  = \mu, \, \text{ and }w_1(\mu - v_1) + w_2(v_2 - \mu)  = d.
\end{split}
\label{eq:problem_2pt_full_equal_mad}
\end{equation}
    \label{thm:reduction_2pt_mad}
\end{theorem}

\section{Fundamental quantities}
\label{sec:fundamental_quantities}
In this section we analyze three fundamental quantities in probability theory involving a random variable $X$, that fit in the framework described in Section \ref{sec:framework}: The conditional expectation, the tail bound and, the so-called (translated) max-operator. Applications building on these quantities will be given in Section \ref{sec:applications}.

\subsection{Conditional expectation}
\label{sec:app_cond}
Consider the conditional expectation $\E(\psi(X) \mid X \geq t)$, where $\psi$ is a concave, increasing function. 
{ If $t \geq \mu$, we can take a sequence of two-point distributions with $v_1 \uparrow \mu$, so that $v_2 \uparrow \infty$. For this sequence, the objective becomes $\lim_{v_2 \rightarrow \infty} \E(\psi(X) \mid X \geq t) = \lim_{v_2 \rightarrow \infty} \psi(v_2)$. It is obvious that this limit cannot be improved upon and hence this is the desired maximum or supremum. }
Therefore, we focus on the case $t < \mu$. In Theorem \ref{thm:cond}, the main result of this section is stated.

\begin{theorem}
Let $\psi$ be a concave, increasing function, and let $t < \mu$. Then
    \begin{equation}\label{eq:v2psiconditional}
        \max_{\P \in \Pmus}\E(\psi(X) \mid X \geq t) = 
        {\color{blue} \psi(v_2)}
    \end{equation}
   with $v_2$ the solution to
    \begin{equation}
        \frac{v_2 - \mu}{v_2 - t}\varphi(t) + \frac{\mu - t}{v_2 - t}\varphi(v_2) = s.
        \label{eq:v2_conditional}
    \end{equation}
    The distribution solving the above problem is the two-point distribution with $v_1 \uparrow t$.
    \label{thm:cond}
\end{theorem}
\begin{proof}
Notice that
$$
\E(\psi(X) | X \geq t) = \frac{1}{\Prob(X \geq t) } \int_{t}^\infty \psi(x) d\Prob(x) = \frac{\E(\psi_u(X))}{\E(\psi_{\ell}(X))} 
$$
for $\psi_{\ell}(x) = \mathbb{I}_{x \geq t}$ and 
$$
    \psi_u(x) = \left\{\begin{array}{ll}
        0  & \text{ if } x < t, \\
        \psi(x) &  \text{ if } x \geq t.
    \end{array}\right. 
    $$
    Note that both $\psi_\ell$ and $\psi_u$ satisfy Assumption \ref{ass:psi}. {Note that if $v_1 \geq t$ then $\E(\psi(X) | X \geq t) = w_1\psi(v_1) + w_2\psi(v_2) \leq \psi(\mu)$ by Jensen's inequality. However if $v_1 < t$ then $\E(\psi(X) | X \geq t) = \psi(v_2) > \psi(\mu)$. Hence the optimum must be attained in the case $v_1 < t$, and so we only have to consider this case when checking Condition \ref{cond:objective}. Since $\psi(v_2)$ is increasing in $v_2$ when leaving $v_1$ constant, this condition is satisfied. Hence,}
    we can apply Theorem \ref{thm:reduction_2pt}, and it suffices now to solve
\begin{equation}
\begin{split}
\mbox{maximize } & \psi(v_2) \\
 \mbox{s.t. } &   w_1 + w_2  = 1, \, w_1v_1 + w_2v_2  = \mu, \, w_1\varphi(v_1) + w_2\varphi(v_2)  = s, \text{ and }
     v_1 < t.
\end{split}
\end{equation}
Proposition \ref{prop:two_point} directly implies that in order to solve this problem, we should send $v_1 \rightarrow t$, and we can then determine $v_2$ by substituting the first and second equation in the third one, and sending $v_1 \uparrow  t$ in the third equation, leading to (\ref{eq:v2psiconditional}).
\end{proof}

 For dispersion defined through the $p$th power deviation, we have the following corollaries.

 \begin{corollary}[$p$th Power deviation]\label{pthpowercondexpec}
     Let $\varphi(x) = |x-\mu|^p$ for $p > 1$ and consider the dispersion constraint $\mathbb{E}[|x-\mu|^p] = s^p$. Then
   \begin{equation}
        \max_{\P \in \Pmup}\E(X \mid X \geq t) = \mu + a,
    \end{equation}
     where $a$ is the solution to 
     \begin{equation}\label{simpeler}
    (\mu-t)\cdot a^p + [(\mu-t)^p - s^p]\cdot a - (\mu-t)\cdot s^p = 0.
\end{equation}
 \end{corollary}
For mean-variance ambiguity $p = 2$ we retrieve
\begin{equation}
        \max_{\P \in \Pmusigma}\E(X \mid X \geq t) = \mu + \frac{\sigma^2}{\mu-t},
    \end{equation}
a result first obtained by Mallows and Richter \cite{mallows1969}. The proof follows from observing that for $p = 2$, equation \eqref{simpeler} reduces to a quadratic equation, whose positive root is given by $a = \sigma^2/(\mu-t)$.

Now consider the MAD case $p = 1$. Note that if $t \geq \mu - s/2$ then the two-point distribution with $v_1 \rightarrow \mu - s/2$ and $v_2 \rightarrow \infty$ yields a conditional expectation that is infinite in the limit. For this reason we assume that $t < \mu - s/2$. We then find the following result:
\begin{theorem}\label{thmmaddd}
For $\mu \in \R$ and $0 \leq d < 2\mu$, it holds that for any $t < \mu - d/2$,
     \begin{equation}\label{eq:thmmaddd}
        \max_{\P \in \Pmud}\E(X \mid X \geq t) = \mu + \frac{(\mu-t)d}{2(\mu-t)-d}.
    \end{equation}
\end{theorem}
\begin{proof}
We wish to apply Theorem \ref{thm:reduction_2pt_mad}. Since we already verified Assumption \ref{ass:psi} we only have to check Condition \ref{cond:mad_condition}, and since we already have Condition \ref{cond:objective} we only have to make verify that $\E(\psi_u(X))/\E(\psi_\ell(X))$ is non-increasing in $v_1$ if $v_2$ is constant. For a fixed $v_2$, we have that
\begin{equation*}
    \frac{\E(\psi_u(X))}{\E(\psi_\ell(X))} = \begin{cases}
        \mu & \mbox{ if } v_1 \geq t, \\
        v_2 & \mbox{ if } v_1 < t.
    \end{cases}
\end{equation*}
Since $v_2 > \mu$ this function is non-increasing in $v_1$, and hence Condition \ref{cond:mad_condition} is satisfied and we can apply Theorem \ref{thm:reduction_2pt_mad}. Proposition \ref{prop:two_point_mad} directly implies that we should sent $v_1 \uparrow  t$, and the expression for $v_2(v_1)$ in  (\ref{eq:madsolution}) then yields (\ref{eq:thmmaddd}). 
\end{proof}
In the recent work \cite{van2023generalized}
an alternative proof of Theorem~\ref{thmmaddd} is presented using the classical primal-dual method for solving semi-infinite fractional programs discussed in the introduction \cite{ji2021data,hettich1993semi,liu2017distributionally}.

\subsection{Tail bound}
\label{sec:app_tail}
In this section we consider the minimization of the tail bound $\Prob(X \geq t)$ for some given value $t$. Using similar reasoning as in Section \ref{sec:app_cond}, it can be argued that if $t \geq \mu$, then the two-point distribution with $v_1 \rightarrow \mu$ (so that $v_2 \rightarrow \infty$) solves the problem, and the tail bound equals zero in the limit. Therefore, we focus on the case $t < \mu$.

\begin{theorem} Let $t < \mu$ and let $\varphi$ satisfy Assumption \ref{ass:phi}. Then,
    \begin{equation}
        \min_{\P \in \Pmus} \P(X \geq t) = \frac{\mu - t}{v_2 -t }
    \end{equation}
 with  $v_2$ the solution to
    \begin{equation}
        \frac{v_2 - \mu}{v_2 - t}\varphi(t) + \frac{\mu - t}{v_2 - t}\varphi(v_2) = s.
        \label{eq:v2_conditional}
    \end{equation}
        The distribution solving the above problem is the two-point distribution with $v_1 \uparrow  t$. 
    \label{thm:tail}
\end{theorem}
\begin{proof}
    The framework in Section \ref{sec:framework} is written for a maximization problem. We can convert the minimization problem at hand to a maximization problem in that framework by setting $$
    \psi_u(x) = \left\{\begin{array}{ll}
        0  & \text{ if } x < t, \\
        -1 &  \text{ if } x \geq t,
    \end{array}\right. 
    $$
    and $\phi_\ell(x) = 1$ for all $x$.
    Then these functions satisfy Assumption \ref{ass:psi}. Furthermore, the function $g$ in Condition \ref{cond:objective} reduces to $1$ in case $v_1 \geq t$, and to $g(v_2) = -(\mu-v_1)/(v_2-v_1)$ when $v_1 < t$, which is non-decreasing in $v_2$. At this point, using a similar argument as in the proof of Theorem \ref{thm:cond} gives the desired result.
\end{proof}

 \begin{corollary}[Power deviation]\label{pthpowertailbound}
    Let $\varphi(x) = |x-\mu|^p$ for $p > 1$ and consider the dispersion constraint $\mathbb{E}[|x-\mu|^p] = s^p$. Then
   \begin{equation}
        \min_{\P \in \Pmus} \P(X \geq t) = \frac{\mu-t}{a + \mu-t},
    \end{equation}
     where $a$ is the solution to
\begin{equation}\label{simpeler2}
 (\mu-t)\cdot a^p + [(\mu-t)^p - s^p]\cdot a - (\mu-t)\cdot s^p = 0.
\end{equation}
\end{corollary}
For $p = 2$, we recover the Cantelli-type inequality
\begin{equation*}
    \min_{\P \in \Pmusigma} \P(X \geq t) = \frac{(\mu - t)^2}{\sigma^2 + (\mu - t)^2}.
\end{equation*}
For $p = 1$, since Condition \ref{cond:mad_condition} is clearly satisfied, we can apply Theorem \ref{thm:reduction_2pt_mad} and with an argument similar to the proof of Theorem \ref{thmmaddd} we find
 $$
 \min_{\P \in \Pmud} \P(X \geq t) = 1 - \frac{d}{2(\mu - t)}
 $$
matching the result from
\cite[Theorem 2]{roos2019chebyshev}.
\subsection{Max-operator}\label{sec:maxoperator}

In this section we study the max-operator $\E(\max\{X-t,0\})$. In Theorem \ref{thm:max} we give the main result of this section.

\begin{theorem}\label{thm:max}
Let $t \in \R$ and let $\varphi$ satisfy Assumption \ref{ass:phi}. The solution to
\begin{align}
    \max_{\P \in \Pmus} \E(\max\{X-t,0\})
\end{align}
is the unique two-point distribution for which $v_1, v_2$ solve the system
    $$
    \begin{array}{c}
        (v_1 - t)\varphi'(v_1) - \varphi(v_1) = (v_2 - t)\varphi'(v_2) - \varphi(v_2),\vspace{0.3cm} \, 
        \displaystyle \frac{v_2 - \mu}{v_2 - v_1}\varphi(v_1) + \frac{\mu - v_1}{v_2 - v_1}\varphi(v_2) = s.
    \end{array}
    $$
\label{thm:max}
\end{theorem}
\begin{proof}
    We can apply Theorem \ref{thm:reduction_2pt} by taking $\psi_\ell(x) = 1$ for all $x \in \R$ and $\psi_u(x) = \max\{x-t,0\}$. It is easily verified (using the fact that $v_1 < t$) that Condition \ref{cond:objective} is satisfied and hence we may apply Theorem \ref{thm:reduction_2pt}. This means that we have to solve the optimization problem
\begin{align*}
	\mbox{maximize } & w_2(v_2 - t) \\
	\mbox{subject to } & w_1 + w_2 = 1,  w_1v_1 + w_2v_2  = \mu, \text{ and } w_1\varphi(v_1) + w_2\varphi(v_2) = s.
\end{align*}
We can solve the first two constraints for $w_1, w_2$ which reduces the problem to
\begin{align*}
	\mbox{maximize } & \frac{(v_2 - t)(\mu - v_1)}{v_2 - v_1} \\
	\mbox{subject to } & \frac{v_2 - \mu}{v_2 - v_1}\varphi(v_1) + \frac{\mu - v_1}{v_2 - v_1}\varphi(v_2) = s.
\end{align*}
To solve this maximization problem we use Lagrange multipliers. The Lagrangian becomes
\begin{equation}
	\mathcal{L}(v_1, v_2, \lambda) = \frac{(v_2 - t)(\mu - v_1)}{v_2 - v_1} + \lambda\left(\frac{v_2 - \mu}{v_2 - v_1}\varphi(v_1) + \frac{\mu - v_1}{v_2 - v_1}\varphi(v_2) - s\right).
\end{equation}
At a local extremum the partial derivatives of $\mathcal{L}$ with respect to $v_1, v_2, \lambda$ must be zero, which implies that
\begin{align*}
	& \frac{(v_2 - \mu)\big(\lambda\big((v_2 - v_1)\varphi'(v_1) - \varphi(v_2) + \varphi(v_1)\big) - v_2 + t\big)}{(v_2 - v_1)^2}  = 0 \\
	& \frac{(\mu - v_1)\big(\lambda\big((v_2 - v_1)\varphi'(v_2) - \varphi(v_2) + \varphi(v_1)\big) - v_1 + t\big)}{(v_2 - v_1)^2}  = 0 \\ 
	& \frac{v_2 - \mu}{v_2 - v_1}\varphi(v_1) + \frac{\mu - v_1}{v_2 - v_1}\varphi(v_2) - s = 0.
\end{align*}
We can reduce the first two equations into a single equation by eliminating $\lambda$, which after some further simplification (using the fact that $v_2 - \mu, \mu - v_1$ and $v_2 - v_1$ are nonzero) yields
\begin{equation}
	(v_1 - q)\varphi'(v_1) - \varphi(v_1) = (v_2 - q)\varphi'(v_2) - \varphi(v_2)
\end{equation}
which means that at this point we have the following implicit definition of the local extremum $v_1, v_2$. Any solution $(v_1, v_2)$ of
\begin{align*}
	f_1(v_1, v_2, t) & := (v_1 - t)\varphi'(v_1) - (v_2 - t)\varphi'(v_2) + \varphi(v_2) - \varphi(v_1) = 0 \\
	f_2(v_1, v_2) & := \frac{v_2 - \mu}{v_2 - v_1}\varphi(v_1) + \frac{\mu - v_1}{v_2 - v_1}\varphi(v_2) - s = 0
\end{align*}
is a local extremum. The question is if the solution to this system is unique, if it exists at all. We first turn to uniqueness. We note that
\begin{align*}
	\frac{\d f_1}{\d v_1}  = -(t - v_1)\varphi''(v_1) < 0, \text{ and }
	\frac{\d f_1}{\d v_2}  = -(v_2 - t)\varphi''(v_2) < 0,
\end{align*}
which shows that if $v_2$ increases then $v_1$ must decrease to maintain $f_1(v_1, v_2, t) = 0$ (and vice versa). However, we also have
\begin{align*}
	\frac{\d f_2}{\d v_1}  = \frac{v_2 - \mu}{v_2 - v_1}\left[\varphi'(v_1) - \frac{\varphi(v_2) - \varphi(v_1)}{v_2-v_1} \right]  < 0, \text{ and }
	\frac{\d f_2}{\d v_2} = \frac{\mu - v_1}{v_2 - v_1}\left[\varphi'(v_2) - \frac{\varphi(v_2) - \varphi(v_1)}{v_2 - v_1}\right] > 0,
\end{align*}
which means that if $v_2$ increases then $v_1$ should also increase to maintain $f_2(v_1, v_2) = 0$ (and vice versa). This makes it impossible to have two different solutions, hence the system has at most one solution. 

Now we consider the boundary cases. Here we need to distinguish between two cases. \\
{Case 1.} $t \leq \mu$. In this case, if $v_1 \rightarrow -\infty$ then the objective converges to $\mu - t$, and if $v_1 \uparrow \mu$ then the objective also becomes $\mu - t$. \\
{Case 2.} $t > \mu$. In this case, if $v_1 \rightarrow -\infty$ then the objective converges to 0, and if $v_1 \uparrow \mu$ then the objective also becomes 0. 

In both cases the boundary values are equal, hence there must be an extremum between them (thereby proving the existence of a solution to the system). And since it is unique, it is either a global maximum or a global minimum. To see that it is a global maximum, we simply show that there exists a two-point distribution with objective strictly higher than the boundary values. In the case $t \leq \mu$, we simply take any two-point distribution with $v_1 < t$. Then the objective is
\begin{equation*}
	w_2(v_2 - t) > w_2(v_2 - t) + w_1(v_1 - t) = (w_2v_2 + w_1v_1) - t(w_1 + w_2) = \mu - t.
\end{equation*}
In the other case, we take any distribution with $v_2 > t$ which yields a strictly positive objective, treating this case as well. We conclude that the unique solution to the system is indeed a maximum for the original optimization problem. 
\end{proof}

If we consider the  $p$th power deviation we obtain the following result.
\begin{corollary}\label{col:max} (Power deviation). Let $\varphi(x) = |x - \mu|^p$ for $p > 1$ and consider the dispersion constraint $\E(|x - \mu|^p) = s^p$, and let $t \in \R$. Then the solution to
\begin{equation*}
    \max_{\P \in \Pmup} \E[\max\{X - t, 0\}]
\end{equation*}
is the unique two-point distribution for $v_1, v_2$ which solve the system
     $$
    \begin{array}{c}
        p(t - v_1)(\mu - v_1)^{p-1} - (\mu - v_1)^p = p(v_2 - t)(v_2 - \mu)^{p-1} - (v_2 - \mu)^p,\vspace{0.3cm} \\ 
        \displaystyle \frac{(v_2 - \mu)(\mu - v_1)}{v_2 - v_1}\big((\mu - v_1)^{p-1} + (v_2 - \mu)^{p-1}\big) = s^p.
    \end{array}
    $$
\end{corollary}

When $p = 2$, we have to solve the system
$$
    \begin{array}{c}
        (v_1 - t)\varphi'(v_1) - \varphi(v_1) = (v_2 - t)\varphi'(v_2) - \varphi(v_2), \text{ and }
        \displaystyle \frac{v_2 - \mu}{v_2 - v_1}\varphi(v_1) + \frac{\mu - v_1}{v_2 - v_1}\varphi(v_2) = s,
    \end{array}
    $$
with $\varphi(x) = x^2$ and $s = \mu^2 + \sigma^2$. The first equation becomes
\begin{align*}
    v_1^2 - 2tv_1  = v_2^2 - 2tv_2 \Longleftrightarrow 
    v_1^2 - 2tv_1 + t^2  = v_2^2 - 2tv_2 + t^2 \Longleftrightarrow
    (v_1 - t)^2  = (v_2 - t)^2
\end{align*}
and since we know that $v_1 < t \leq v_2$, we find that $v_2 - t = t - v_1$, or $v_1 = 2t - v_2$. The second equation becomes
$    (v_2 - \mu)(\mu - v_1) = \sigma^2 $
and after plugging in $v_1 = 2t - v_2$ we find
$    v_2^2 - 2tv_2 + 2t\mu - \mu^2 \sigma^2 = 0 $
with solution $v_2 = t + \sqrt{(t - \mu^2) + \sigma^2}$ (note that the other solution is less than $t$ and hence not feasible). This yields $v_1 = t - \sqrt{(t - \mu)^2 + \sigma^2}$, matching the classical result of \cite{scarf1958min}. \\

For the MAD case $p = 1$, we find the following result.
\begin{theorem}\label{thm:maxoperatormad}
    Let $t \in \R$ and $\varphi(x) = |x - \mu|$. Then the solution to
    \begin{align*}
        \max_{\P \in \Pmud} \E[\max\{X-t,0\}]
    \end{align*}
    can be described as follows.
    \begin{enumerate}
        \item[{\rm (i)}] If $\mu - t > 0$, the optimal (limiting) distribution is obtained by taking  $v_1 \rightarrow -\infty$, $v_2 \rightarrow \mu + d/2$, with objective  $\mu - t + d/2$. 
        \item[{\rm (ii)}]  If $\mu - t < 0$, the optimal distribution is obtained by taking $v_1\rightarrow \mu - d/2$, $v_1 \rightarrow \infty$,  with objective is $d/2$. 
        \item[{\rm (iii)}]  If $\mu - t = 0$,  any distribution for $X$ yields the same objective $d/2$. 
    \end{enumerate}
\end{theorem}
\begin{proof}
The function $g(v_1, v_2)$ in (\ref{eq:cond_objective}) becomes
\begin{equation}
    g(v_1, v_2) = w_2(v_2 - t) = \frac{(\mu - v_1)(v_2 - t)}{v_2 - v_1}
\end{equation}
and (using the fact that $v_2 > t > v_1$ and $v_2 > \mu > v_1$) we find that
\begin{equation*}
    \frac{\d g}{\d v_1} = -\frac{(v_2 - t)(v_2 - \mu)}{(v_2 - v_1)^2} < 0, \quad \frac{\d g}{\d v_2} = \frac{(t - v_1)(\mu - v_1)}{(v_2 - v_1)^2} > 0.
\end{equation*}
This means that Condition \ref{cond:mad_condition} is satisfied, so we can apply Theorem \ref{thm:reduction_2pt_mad}. We find that
\begin{equation}\label{eq:tomaximize}
    w_2(v_2 - t) = \frac{d}{2}\left(1 + \frac{\mu - t}{v_2 - \mu}\right)
\end{equation}
If $\mu - t > 0$, then this is maximized if $v_2 - \mu$ is minimized: hence if $v_2 \rightarrow \mu + d/2$ and $v_1 \rightarrow -\infty$. The optimal objective in this case is $\mu - t + d/2$. If $\mu - t < 0$, then (\ref{eq:tomaximize}) is maximized if $v_2 - \mu$ is maximized, hence if $v_2 \rightarrow \infty$ and $v_1 \rightarrow \mu - d/2$. The optimal objective is then $d/2$. If $\mu - t = 0$, then (\ref{eq:tomaximize}) is independent of $v_2$, and so every value $v_2$ yields the same objective $d/2$. In fact, it is easy to see that every distribution yields this objective. 
\end{proof}
The upper bound for the expected maximum operator in Theorem \ref{thm:maxoperatormad} also follows from results in \cite[Theorem 3]{bental1972mad}, although in that paper the worst-case distribution is a degenerate three-point distribution. 

\subsection{Numerical results}
We now present some numerical results for the tight bounds, and assess the interplay between the bounds and the level of ambiguity by considering as dispersion measure the $p$th power deviation for various values of $p$. 
Figure \ref{fig:condtailplot2} shows the bounds for the conditional expectation, tail bound and max operator using Corollaries \ref{pthpowercondexpec}, \ref{pthpowertailbound} and \ref{col:max}, respectively.
\begin{figure}[!ht]
\centering
\includegraphics[width=.48\textwidth]{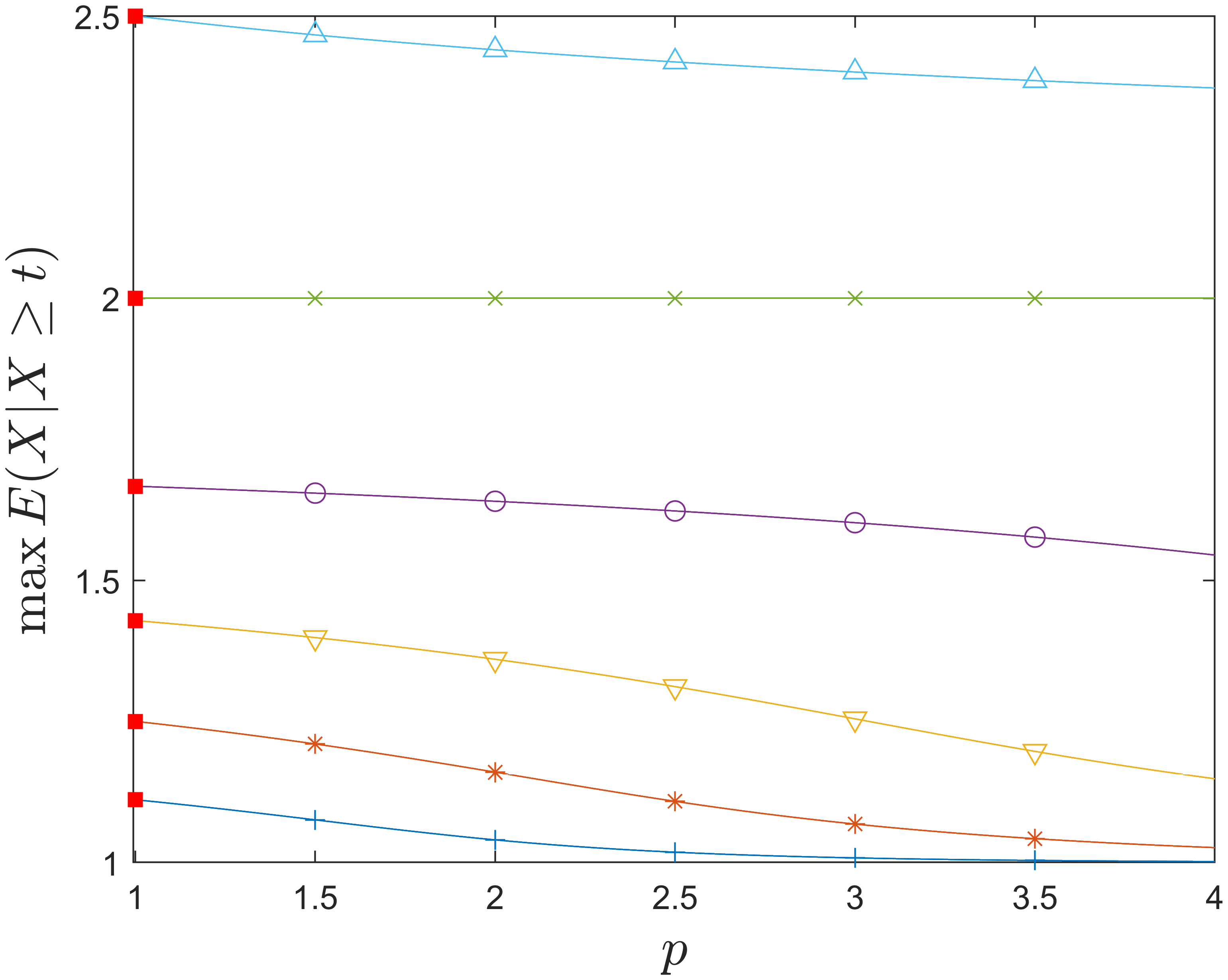}\hfill
\includegraphics[width=.48\textwidth]{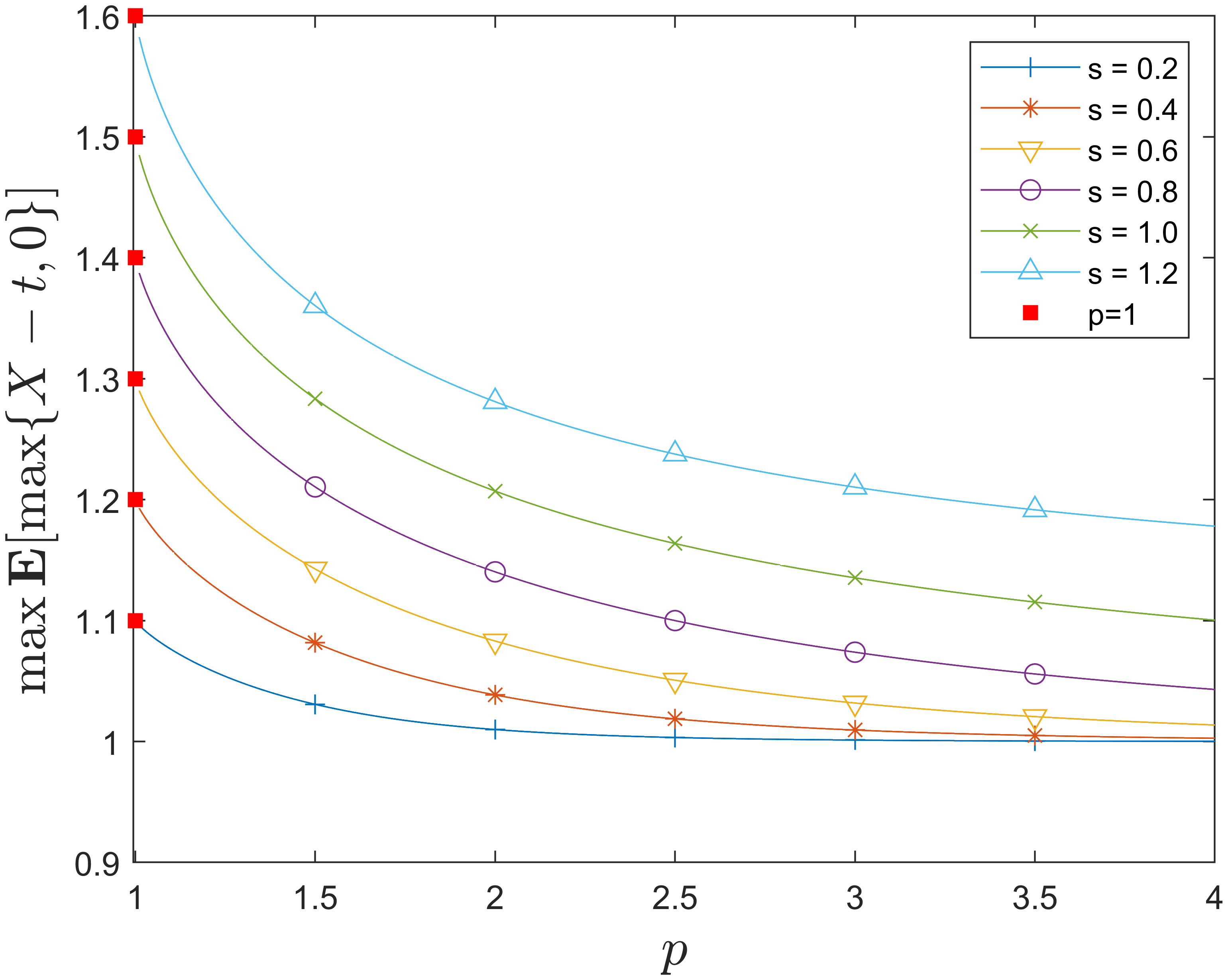}
\caption{The conditional expectation and max operator for a selected set of values $s$, for $\mu = 1$ and $t = 0$.  Note that for $s = 1$ the conditional expectation is constant, since the optimal distribution in this case is independent of the value of $s$. Overall though, the conditional expectation and max operator are increasing in $s$ and decreasing in $s$. 
Also note that for the max operator, the curves for small $s$ converge much faster to 1 as $p$ gets large than for large $s$. }
\label{fig:condtailplot2}
\end{figure}
Observe that for $s = 1$ the conditional expectation and tail bound are both insensitive to $p$, because the optimal distribution
is independent of $p$. For all other values $s$, the conditional expectation decreases with $p$ and increases in $s$. 
To see why, notice that low $p$ and high $s$ correspond with less restrictive ambiguity sets, and hence higher conditional expectation.
Similarly, the tail bound increases with $p$ and decreases with $s$. 

For the max operator, a highly restrictive ambiguity set with $\mu > t$ would mean that $\max\{X - t, 0\}$ is almost always $X - t$, so $\E(\max\{X - t, 0\})\approx \E(X - t) = \mu - t$, which in Figure \ref{fig:condtailplot2} is equal to $1 - 0 = 1$. On the other side, an ambiguity with only the constraint $\E(X) = \mu$ but no constraints on dispersion would mean that one could choose $v_1 = \mu - R$ and $v_2 = \mu + R$ with $R$ extremely large, so that $\E(\max\{X - t, 0\}) = \tfrac{1}{2}(\mu + R)$, which tends to infinity as $R \rightarrow \infty$. Hence in this case the max operator tends to infinity. So again we have that a more restrictive ambiguity set (i.e. a lower $s$ or higher $p$) should imply a lower version of the max operator, which is consistently the case here. Note that there is no $s > 0$ for which the graph is constant, unlike in the other two problems we considered: for any fixed value $s > 0$ the graph converges to 1 if $p$ becomes  large.

\section{Applications}
\label{sec:applications}

We will now leverage the tight bounds obtained in 
Section \ref{sec:fundamental_quantities} for the conditional expectation, the tail bound and the max-operator for performing  distribution-free analysis of the newsvendor model and monopoly pricing.


\subsection{Robust newsvendor}
The newsvendor model is a classical model in inventory management used for determining the optimal order quantity for perishable goods or goods with uncertain demand. The central idea is to balance the cost of holding excess inventory against the cost of stockouts. The objective function is the expected total cost which can contain the max operator of the market demand. The distribution-free problem then minimizes expected cost, considering all demand scenarios in the ambiguity set, with as solution a robust order quantity that protects against the demand uncertainty. 
The newsvendor chooses to stock $q$ items, while the demand follows a stochastic variable $D$. For every item less than the demand the newsvendor faces a penalty $b$, while it faces penalty $h$ for holding any overstocked item. The newsvendor then needs to consider the cost function
$C(q)=b\max\{D-q,0\}+h\max\{q-D,0\}$.
Since $\max\{q-D,0\}=q-D+\max\{D-q,0\}$ we can write 
\begin{align}\label{costnews}
\mathbb{E}C(q)=h(q-\mathbb{E}D)+(b+h)\mathbb{E}(\max\{D-q,0\})
\end{align} 
and the optimal $q=q^*$ thus solves
\begin{align}
\min_{q} \ hq+(b+h)\mathbb{E}(\max\{D-q,0\}).
\end{align}
The newsvendor model thus uses as input the max operator $\mathbb{E}_{\P}(\max\{D-q,0\})$, which is also of central importance in risk management in finance and insurance mathematics where a stop-loss order is an order placed with a broker to buy or sell once the stock $D$ reaches a certain price.
We consider the robust newsvendor, which means that after the newsvendor chooses how many items to stock, nature will choose the worst distribution for $D$ in the ambiguity set
for this $q$. That is
\begin{align}
\min_{q} \ hq+(b+h)\max_{\P \in \Pmus}\mathbb{E}(\max\{D-q,0\}).
\end{align}
Note that the maximization over $\P$ follows from Theorem \ref{thm:max}. For mean-variance, this maximization problem can be solved exactly, and after that the minimization over $q$ can also be done analytically. Scarf \cite{scarf1958min} showed that 
\begin{align}
q^\star=\mu+\frac{\sigma}{2}\Big(\sqrt{\frac{b}{h}}-\sqrt{\frac{h}{b}}\Big),
\end{align}
with the extremal two-point distribution 
\begin{align}
D^{(2)}=\left\{\begin{array}{ll}
\mu-\sigma\sqrt{\frac{h}{b}}&  {\rm w.p.} \  \frac{b}{b+h},\\
\mu+\sigma\sqrt{\frac{b}{h}}& {\rm w.p.} \  \frac{h}{b+h},
  \end{array}\right.
\end{align}
and $\mathbb{E}C(q^\star) = \mu + \sigma\sqrt{bh}$. For $p = 1$ the optimal order quantity is any quantity $q \geq \mu$. For simplicity we can just take $q^\star = \mu$, which we also do in Figure \ref{fig:robustnewsvendor}.  The optimal objective is then
\begin{equation}
    \E C(q^\star) = h\mu + (b + h)s/2.
\end{equation}
We shall now explore how heavy-tailed demand affects the robust order quantity for $p \in (1,2)$. 
This requires careful computations, as the optimization over the unknown distribution $\Prob$ has no closed-form solutions for $p \in (1,2)$, which in turn makes the optimization over $q$ more challenging. We will leverage Theorem \ref{thm:max}, which says we can restrict to the two-point distributions  that are defined implicitly in Corollary \ref{col:max}. For a fixed $p$ we  compute the cost function for several $q$-values to find the optimal order quantity. Figure \ref{fig:robustnewsvendor} shows results for  $\mu = 1, h = 1, b = 10$.

\begin{figure}[!ht]
    \centering
    \includegraphics[width=0.48\textwidth]{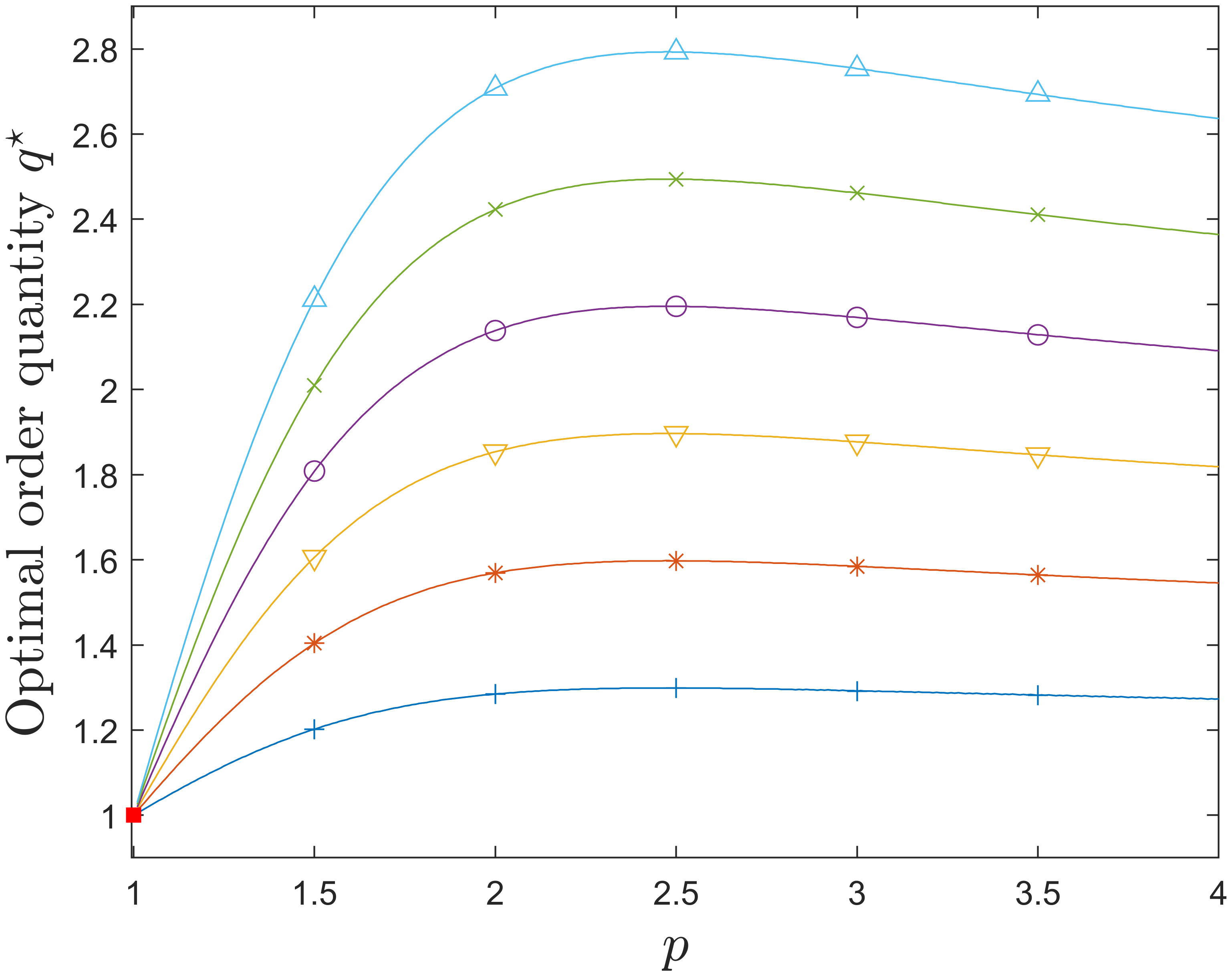}\hfill
    \includegraphics[width=0.48\textwidth]{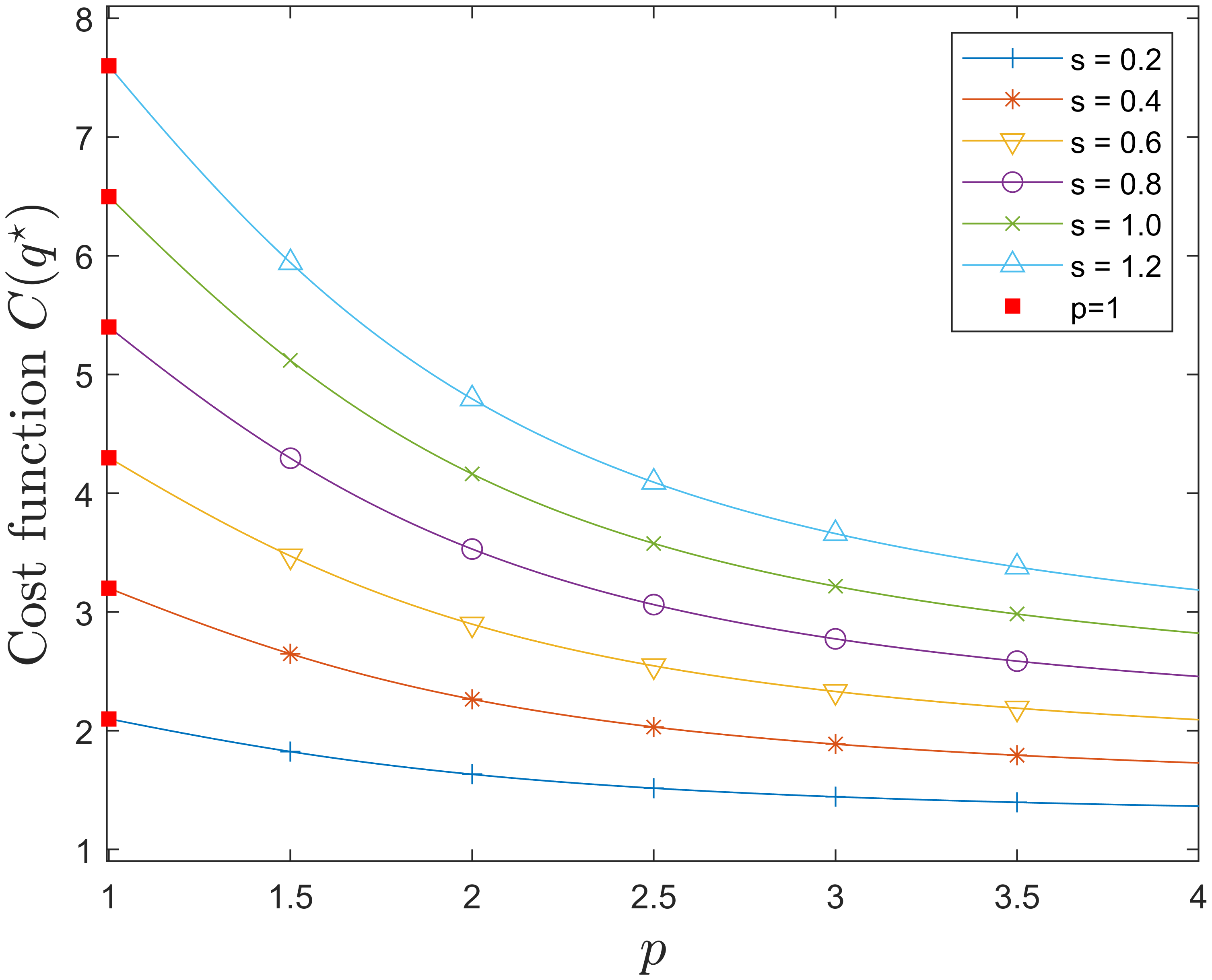}
    \caption{The optimal order quantity for the robust newsvendor, now with $\mu = 1, b = 10, h = 1$. Note that at $\bar{p} \approx 2.5$ there is an extremum in all curves: for $p < \bar{p}$ we have that $q^\star$ increases with $p$, but for $p > \bar{p}$ instead $q^\star$ decreases with $p$. The cost function $C(q^\star)$ is always decreasing in $p$ though. }
    \label{fig:robustnewsvendor}
\end{figure}

Observe that the optimal order quantity  decreases with $s$. This can be understood by arguing that for $b > h$ it is always better to order more than $\mu$ items, since we gain more by selling an extra item than what we lose by having an unsold item. 
This effect becomes more pronounced for higher $s$, which implies that the probability of demand exceeding the mean becomes larger as well.

Arguably more surprising is the dependence between $q^\star$ and $p$. A first thing to notice is that for $p \downarrow 1$,  the optimal order quantity always approaches $\mu$. One might expect a larger $p$ implies less dispersion and hence a lower $q^\star$. But $q^\star$ turns out to be non-monotonic in $p$: the order quantity first increases and then decreases in $p$. This effect is confirmed if we plot $C(q)$ as a function of $q$ for several values $p$, see Figure \ref{fig:robustnewsvendorexperiment}. Let $\bar{p}$ denote the value of $p$ at which $q^\star$ turns from increasing to decreasing.
Based on extensive numerical experiments,
we have seen that $\bar{p}$ depends on the ratio $b/h$; Table~\ref{tablepbar} shows that $\bar{p}$ decreases with $b/h$.
\begin{table}[h]\label{tablepbar}
\begin{center}
    \begin{tabular}{|c||c|c|c|c|c|c|c|c|c|c|}\hline
        $b/h$ & 10 & 15 & 20 & 25 & 30 & 40 & 50 & 60 & 70 & 80 \\ \hline
        $\bar{p}$ & 2.48 & 2.13 & 1.96 & 1.86 & 1.79 & 1.70 & 1.64 & 1.60 & 1.57 & 1.54 \\ \hline
    \end{tabular}
\end{center}
\caption{Dependence between $\bar{p}$ and $b/h$.}
\end{table}
We conclude that allowing \blu{distributions with heavier tails} always increases the  costs $C(q^\star)$, but does not necessarily lead to a smaller order quantity $q^\star$. 
\begin{figure}[!ht]
    \centering
    \includegraphics[width=0.48\textwidth]{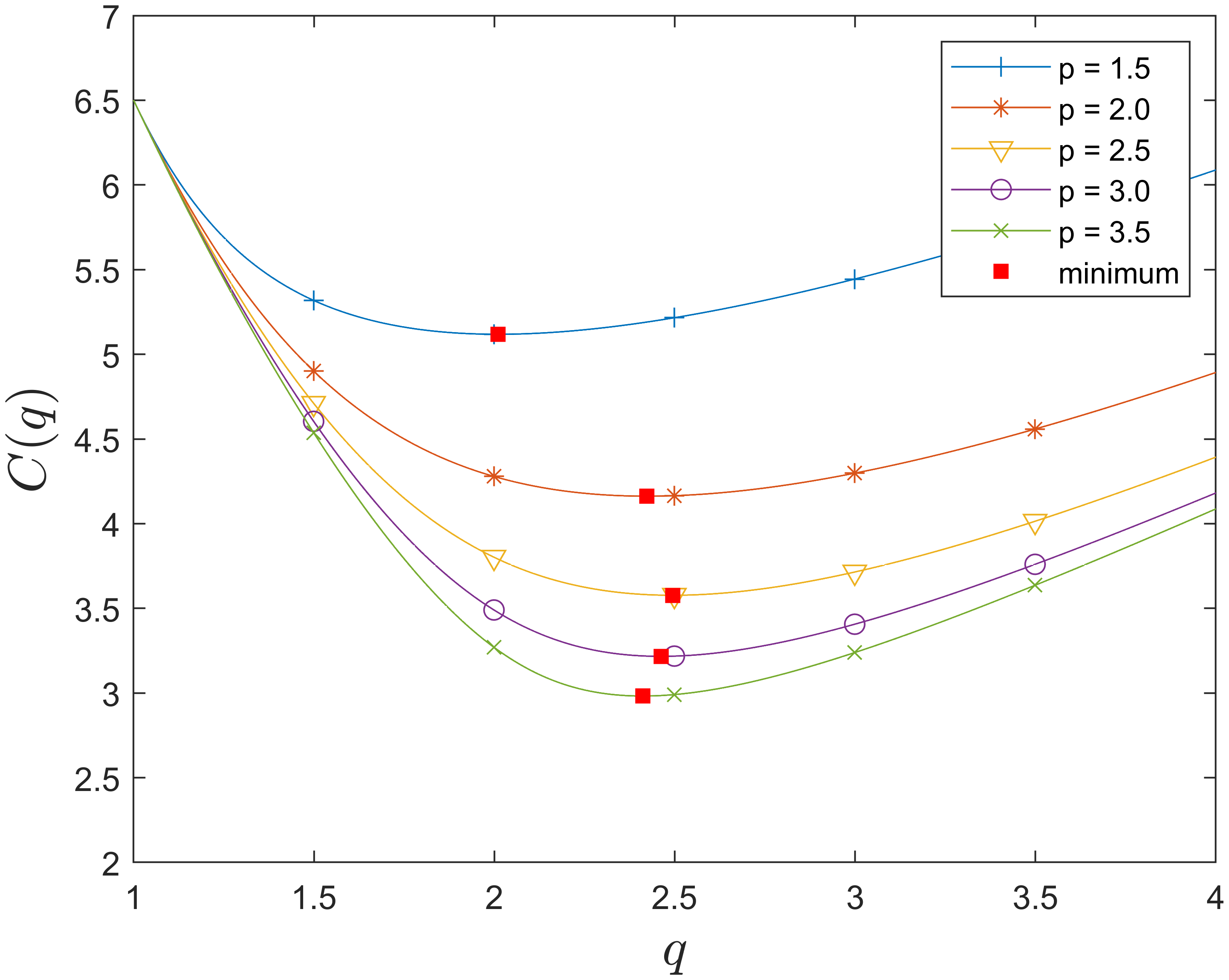}\hfill
    \caption{The function $C(q)$ as a function of $q$ for several values $p$, with the minimum (corresponding to $q^\star$) marked in all curves. It is clear that $q^\star$ is indeed increasing for low $p$ but decreasing for high $p$, confirming the effect we saw in Figure \ref{fig:robustnewsvendor}.}
    \label{fig:robustnewsvendorexperiment}
\end{figure}


\color{black}

\subsection{Robust pricing}
In the monopoly pricing problem, a seller wants to sell an item with demand distribution $\Prob$ for a fixed price $r$. The goal of the seller is to choose a price $r$ that maximizes the expected revenue $\REV(r,\Prob) = r \Prob(D \geq r)$, where the random demand $D \sim \Prob$ is known to the seller. This is a classical problem in the mechanism design literature \cite{riley1983optimal,myerson1981optimal}. We study the problem from a robust perspective, where first the seller has to choose a selling price $r$, after which nature gets to choose a distribution from a given ambiguity set that is worst-case for the price chosen by the seller.

In this case there are various revenue objectives that one can consider; see \cite{azarmicali2013,chen2022distribution,chen2021screening}. We focus on the approximation ratio that was first studied by Giannakopoulos et al.~\cite{giannakopoulos2023robust} for mean-variance ambiguity. 
For any pair $(r, \P)$  define 
$\OPT(\P) = \max_{r'} \REV(r', \P)$ as the optimal revenue the seller can generate for this probability distribution. As a measure for how good the price $r$ chosen by the seller is, we consider the ratio
\begin{equation}\label{eq:ratio}
	\frac{{\rm OPT}(\P)}{{\rm REV}(r, \P)}.
\end{equation}
A small approximation ratio (close to 1) indicates that the robust prices provides a solution that is close to optimal, while a larger ratio implies a less accurate approximation. The seller wishes to minimize the ratio in the face of uncertain market conditions, and hence the optimal price can be found by solving the distribution-free minimax problem
\begin{equation}\label{eq:ratio}
	\min_r \max_{\P \in \mathcal{P}(\mu, s,\varphi)}\frac{{\rm OPT}(\P)}{{\rm REV}(r, \P)}.
\end{equation}
It turns that the analysis of this objective is intricately related to the robust conditional expectation and tail bound results derived in Section \ref{sec:fundamental_quantities}. This is summarized in \blue{Proposition \ref{prop:robustpricing}. The statement in this proposition remains true if the distributions in the ambiguity set are supported on $\R_{\geq 0}$ instead of $\R$. This follows because the results in Sections \ref{sec:app_cond} and \ref{sec:app_tail} still hold for distributions with nonnegative support. This is due to our reduction technique, which preserves nonnegativeness when reducing to two points, and the fact that the worst-case distribution for the unrestricted case has nonnegative support when $t \geq 0$. }

\begin{proposition}\label{prop:robustpricing}
Let $r < \mu$, then
    $$
    \max_{\P \in \mathcal{P}(\mu, s, \varphi)}\frac{\textup{OPT}(\P)}{\textup{REV}(r, \P)} = \max\left\{ \frac{1}{\min_{\P \in \mathcal{P}(\mu, s, \varphi)} \P(X \geq r)}, \frac{\max_{\P \in \mathcal{P}(\mu, s,\varphi)} \E(X | X \geq r)}{r}\right\}.
    $$
        The optimal strategy for nature is to choose the unique two-point distribution supported on $v_1$ and $v_2$ in the ambiguity set, with $v_1  \rightarrow r^-$, so that 
        $$
        \max_{\P \in \mathcal{P}(\mu, s, \varphi)}\frac{\textup{OPT}(\P)}{\textup{REV}(r, \P)} = \max\left\{\frac{\mu-r}{v_2-r},\frac{v_2}{r}\right\}.
        $$
        where $v_2$ satisfies
        \begin{equation*}
        \frac{v_2 - \mu}{v_2 - r}\varphi(r) + \frac{\mu - r}{v_2 - r}\varphi(v_2) = s.
        \label{eq:v2_conditional}
    \end{equation*}
\end{proposition}
\begin{proof}
    We follow the proof outline of \cite{giannakopoulos2023robust} who treat the case of mean-variance ambiguity. Let $r^* = r^*(\P)$ be the optimal selling price, given a distribution $\P$, so that $\OPT(\P) = r^*\P(X \geq r^*)$. 
    We next fix $\P$, and consider two cases.

    {Case 1: $r^* \leq r$.} Then
    $$
    \frac{\OPT(\P)}{\REV(r, \P)} = \frac{ r^*\P(X \geq r^*)}{ r\P(X \geq r)} \leq \frac{1}{\P(X \geq r)} \leq \frac{1}{\min_{\P \in \mathcal{P}(\mu, s,\varphi)} \P(X \geq r)}.
    $$
    In the first inequality, we use that $r^*/r \leq 1$ and $\P(X \geq r^*) \leq 1$.

    {Case 2: $r^* > r$.} Following the same line of reasoning as in \cite[Eq. 6]{giannakopoulos2023robust}, we find that in this case
    $$
    \frac{ r^*\P(X \geq r^*)}{\P(X \geq r)} \leq \E(X | X \geq r),
    $$
    and so
    $$
    \frac{\OPT(\P)}{\REV(r, \P)} \leq \frac{ \E(X | X \geq r)}{r} \leq \frac{ \max_{\P \in \mathcal{P}(\mu, s,\varphi)}\E(X | X \geq r)}{r}.
    $$
    Taking these cases together, and observing that $\P$ was chosen arbitrarily, we find
    $$
     \max_{\P \in \mathcal{P}(\mu, s)}\frac{\OPT(\P)}{\REV(r, \P)} \leq \max\left\{ \frac{1}{\min_{\P \in \mathcal{P}(\mu, s,\varphi)} \P(X \geq r)}, \frac{\max_{\P \in \mathcal{P}(\mu, s, \varphi)} \E(X | X \geq r)}{r}\right\}.
    $$
    We next argue that for the two-point distribution supported on $\{v_1,v_2\}$, with $v_1 \rightarrow r^-$, the quantity on the right is in fact attained.    
    Let $\P^*$ denote the distribution where $v_1 \rightarrow r^-$. Then $\OPT(\P^*) = \max\{v_1,w_2v_2\} = \max\{r,w_2v_2\}$, and so
\begin{equation}\label{eq:prop41intermediate}
        \frac{\OPT(\P^*)}{\REV(r, \P^*)} = \frac{\max\{r,w_2v_2\}}{rw_2} = \max\left\{\frac{1}{w_2},\frac{v_2}{r}\right\}.
    \end{equation}
    We know from other results that $w_2 = \min_{\P \in \mathcal{P}(\mu, s,\varphi)} \P(X \geq r)$ and $v_2 = \max_{\P \in \mathcal{P}(\mu, s,\varphi)} \E(X | X \geq r)$ when $v_1 \rightarrow r^-$. That is, the two-point distribution that sends its left support point towards $r$ solves both the problems of minimizing the tail bound, as well as maximizing the condition expectation. Substituting $w_2 = (\mu - v_1)/(v_2-v_1)$ with $v_1 \rightarrow r^-$ then gives the desired result.
\end{proof}

\color{black}
Using Proposition \ref{prop:robustpricing} (or specifically the intermediate result (\ref{eq:prop41intermediate}), the robust pricing problem (\ref{eq:ratio}) becomes
\begin{equation}\label{eq:ratio2}
    \min_r \max_{\P \in \mathcal{P}(\mu, s,\varphi)}\left\{\frac{1}{w_2}, \frac{v_2}{r}\right\}.
\end{equation}
The question is now for which $r$ this minimum is attained. It is clear that $v_2$ is increasing in $v_1$ (and hence in $r$), hence $w_2$ is decreasing in $r$ and $1/w_2$ is increasing in $r$. However $v_2/r$ converges to infinity if $r \downarrow 0$, and also if $v_2 \uparrow \mu$. The minimum over all $r$ is therefore either attained at the intersection point of the curves $1/w_2$ and $v_2/r$, or at the lowest point of $v_2/r$ for $0 < r < \mu$. 

The mean-variance case was analyzed in \cite{giannakopoulos2023robust}. 
\begin{corollary}[\cite{giannakopoulos2023robust}]
    For the mean-variance ambiguity set, the optimal value $r$ is always at the intersection point of $1/w_2$ and $v_2/r$. 
\end{corollary}
This raises the question whether the intersection point is also always optimal for other ambiguity sets. This turns out not to be the case. To see this, consider the mean-MAD ambiguity set, for which the following result holds:\begin{theorem}
    Let $\mathcal{P}(\mu, \delta\mu)$ be the ambiguity set of all distributions with mean $\mu$ and MAD $\delta\mu$, where $0 \leq \delta < 2$. Then the optimal robust selling price $\rho = r/\mu$ solving \eqref{eq:ratio} is
    \begin{equation}
    \rho = \begin{cases}
    \frac{4 + \delta - \sqrt{\delta^2 + 8\delta}}{4} & \mbox{ if } 0 \leq \delta \leq 2(\sqrt{5} - 2), \\
    \frac{2 - \delta}{2 + \delta} & \mbox{ if } 2(\sqrt{5} - 2) \leq \delta < 2,     
        \end{cases}
    \end{equation}
    and the resulting ratio is
    \begin{equation}
        \min_{\rho}\max_{\P \in \mathcal{P}(\mu, \delta\mu)}\frac{{\rm OPT}(\P)}{{\rm REV}(\rho\mu, \P)} = 1 + \begin{cases}
            \frac{2\delta}{4\sqrt{\delta^2 + 8\delta} - 3\delta} & \mbox{ if } 0 \leq \delta \leq 2(\sqrt{5} - 2), \\
            \frac{8\delta}{(2 - \delta)^2} & \mbox{ if } 2(\sqrt{5} - 2) \leq \delta < 2.
        \end{cases}
    \end{equation}
\end{theorem}
\begin{proof} For the mean-MAD ambiguity set, it is known that
$$    \frac{1}{w_2} = \frac{2(\mu - r)}{2(\mu - r) - d}, \text{ and }
    \frac{v_2}{r}  = \frac{2\mu(\mu - r) - dr}{r(2(\mu - r) - d)}.
    $$
The optimal $r$ is either at the intersection point $r_1$ of $1/w_2$ and $v_2/r$, or at the value $r_2$ that minimizes $v_2/r$. We can compute $r_1$ by solving $1/w_2 = v_2/r$ and $r_2$ by minimizing $v_2/r$:
$$
    r_1  = \mu + \frac{1}{4}d - \frac{1}{4}\sqrt{d^2 + 8\mu d}, \text{ and }
    r_2  = \frac{2\mu^2 - \mu\delta}{2\mu + \delta},
$$
or, with the scale-free $\rho = r/\mu$ and $\delta = d/\mu$
$$
    \rho_1  = \frac{4 + \delta - \sqrt{\delta^2 + 8\delta}}{4}, \text{ and }
    \rho_2  = \frac{2 - \delta}{2 + \delta},
$$
neither of which is always greater than the other. For $\delta < 2(\sqrt{5} - 2)$ we have $\rho_1 < \rho_2$ (hence $\rho_1$ is optimal), but if $\delta > 2(\sqrt{5} - 2)$ we have $\rho_2 < \rho_1$ (hence $\rho_2$ is optimal). The resulting optimal ratio can be directly computed from this. \end{proof}
The above results for both mean-variance and mean-MAD ambiguity  are visualized in Figure~\ref{fig:ratiofigures}. For mean-MAD ambiguity, the intersection point marks the transition between the regime where $\rho_1$ resp. $\rho_2$ is optimal. 
\begin{figure}[!ht]
    \centering
    \includegraphics[width=0.48\textwidth]{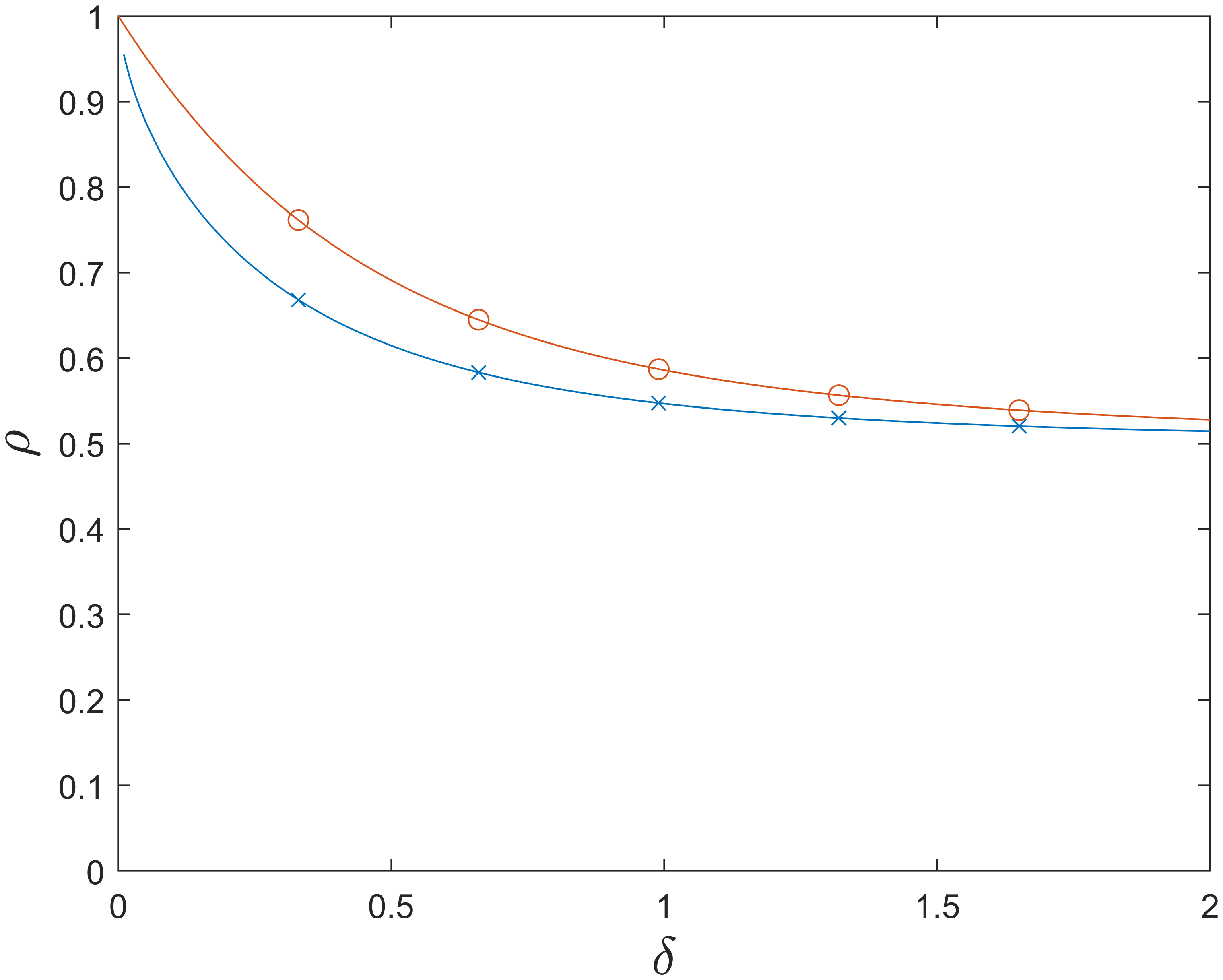}\hfill
    \includegraphics[width=0.48\textwidth]{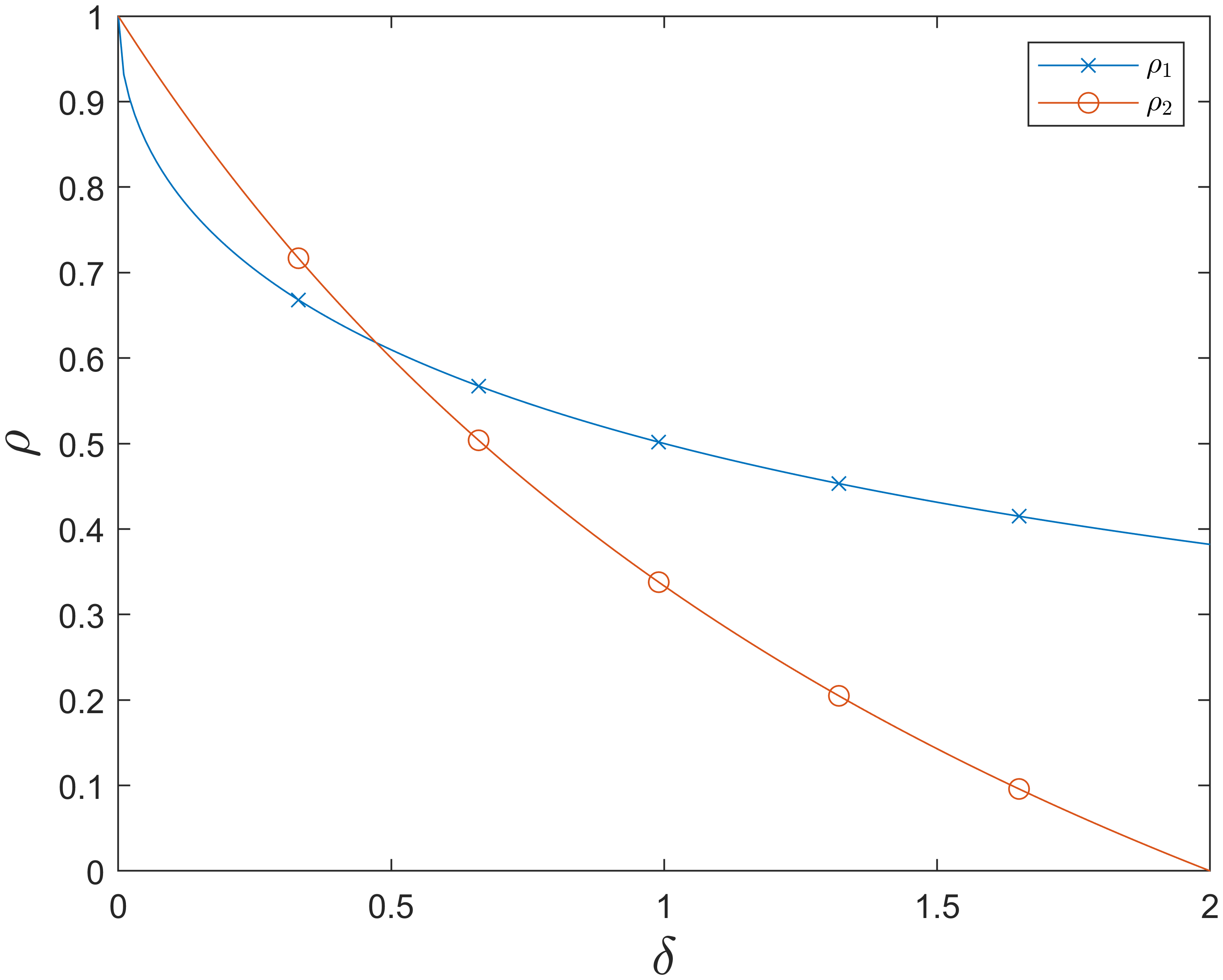}
    \caption{The two candidates $\rho_1$ (intersection point) and $\rho_2$ (minimum of $v_2/r$) visualized for both mean-variance ambiguity (left) and mean-MAD ambiguity (right). For mean-variance, $\rho_1$ is always smaller and hence optimal. However, for mean-MAD, the $\rho_1$ is optimal for $\delta < 2(\sqrt{5} - 2)$, while $\rho_2$ is optimal otherwise.}
    \label{fig:ratiofigures}
\end{figure}
But what happens then for an ambiguity set with the $p$th power deviation constraint, with $1 < p < 2$? Although in this case we cannot compute in closed form the exact values $\rho_1$ and $\rho_2$, we can evaluate these values numerically. Based on extensive numerical experiments, we conclude that for any $p < 2$ there is a transition value $t(p)$, such that if $\delta < t(p)$ then $\rho_1$ is optimal and if $\delta > t(p)$ we have that $\rho_2$ is optimal. If $p \uparrow 2$ then we found $t(p) \uparrow \infty$, matching the fact that for $p = 2$ there is no transition anymore. Taken together, this leads to the following conjecture:
\begin{conjecture}\label{con:ratio}
Let $\varphi(x) = |x - \mu|^p$ for $p > 1$ and consider the dispersion constraint $\E(|x - \mu|^p) = s^p$. For the optimal ratio, there are two candidates, namely the intersection point $\rho_1$ where $1/w_2 = v_2/r$ and the minimum $\rho_2$ of $v_2/r$:
\begin{enumerate}
    \item[{\rm (i)}] If $1\leq p < 2$, then there is a finite transition value $t(p)$, such that if $\delta < t(p)$ then $\rho_1 < \rho_2$ (i.e. $\rho_1$ is optimal), while if $\delta > t(p)$ then $\rho_2 < \rho_1$ (i.e. $\rho_2$ is optimal). 
    \item[{\rm (ii)}]  If $p \geq 2$ then 
    $\rho_1 < \rho_2$ for any $\delta$ (i.e. $\rho_1$ is always optimal).
\end{enumerate}
\end{conjecture}
Settling the conjecture is challenging, because
$\rho_1$ and $\rho_2$ are defined as implicit functions.

\section{Conclusion and future research}\label{conclus}

We have introduced a new method for obtain distribution-free bounds for a class of expectation operators. The method consists of solving non-linear optimization problems by a reduction techniques that reduces the set of all candidate solutions to two-point distributions. This presents a transparent and tractable alternative for more general methods that use primal-dual techniques. 

The two-point bounds were then applied to the  newsvendor model and monopoly pricing, both settings that involve decision-making under uncertainty. The newsvendor model focuses on inventory decisions, while monopoly pricing focuses on pricing decisions. The common thread when applying DRO is the application of minimax or maximin principles to make decisions that balance costs and revenues in the face of uncertainty. \blue{We contribute to  DRO by considering ambiguity sets that restict to distributions with relatively lighter tails, or by allowing for distributions with heavier tails.} Indeed, when one wants to consider among all possible scenarios also extreme scenarios, the ambiguity set should be designed to include distributions with \blue{fairly extreme dispersion}. The sets in this paper encompass a range of possible tail behaviors, by conditioning on the dispersion measure, capturing the uncertainty about extreme events. We have shown how such \blue{extreme scenarios} influence the robust decisions of the newsvendor and the monopolist. 

Apart from settling Conjecture~\ref{con:ratio}, it seems worthwhile to broaden the scope of application of our novel method by finding more cases that fit in the general optimization problem \eqref{richclass}.

\paragraph{Acknowledgements} This work was been supported by the NWO Vici grant 202.068 of Johan S.H.~van Leeuwaarden.




  \bibliographystyle{plain} 
  \bibliography{references}





\end{document}